\documentclass[journal,10pt]{elsarticle}
\usepackage[OT1]{fontenc}  	
\usepackage[utf8]{inputenc}
\usepackage[greek,swedish,spanish,french,english,USenglish]{babel} 		
\usepackage{amssymb}  			
\usepackage{amsmath}	
\usepackage{amsthm}	
\usepackage{mathrsfs}	
\usepackage{esint}
\usepackage{units}
\usepackage{bbm}   				
\usepackage{color,graphicx} 			
\usepackage[small]{caption}
\usepackage{cases}			
\usepackage{mathpazo}  
\usepackage{bm}  

\newcommand{\be}{\begin{equation}}
\newcommand{\ee}{\end{equation}}

\renewcommand{\Re}{\mathop{\rm Re}\nolimits}
\newcommand{\cas}{\mathop{\rm cas}\nolimits}
\newcommand{\sh}{\mathop{\rm sh}\nolimits}
\newcommand{\ch}{\mathop{\rm ch}\nolimits}
\renewcommand{\th}{\mathop{\rm th}\nolimits}

\newcommand{\tg}{\mathop{\rm tg}\nolimits}
\newcommand{\ctg}{\mathop{\rm ctg}\nolimits}

\newcommand{\pv}{\mathop{\rm p.v.}\nolimits}

\newcommand{\DHT}{\mathop{\rm DHT}\nolimits}

\newcommand{\sgn}{\mathop{\rm sgn}\nolimits}

\newtheorem{theorem}{Theorem}
\newtheorem{lemma}{Lemma}
\newtheorem{corollary}{Corollary}

\newtheorem{remark}{Remark}

\makeatletter
\def\ps@pprintTitle{%
\let\@oddhead\@empty
\let\@evenhead\@empty
\let\@oddfoot\@empty
\let\@evenfoot\@oddfoot
}\makeatother

\begin{document}

\begin{frontmatter}

\title{On a finite sum of cosecants appearing in various problems}

\author[mymainaddress,mysecondaryaddress]{Iaroslav V.~Blagouchine\corref{cor1}} 
\ead{iaroslav.blagouchine@univ-tln.fr}
\author[mysecondaryaddress]{Eric Moreau} 
\ead{moreau@univ-tln.fr}
\address[mymainaddress]{Steklov Institute of Mathematics at St. Petersburg, Russia.}
\address[mysecondaryaddress]{University of Toulon, France.} 
\cortext[cor1]{Corresponding author.}

\begin{abstract}
Finite trigonometric sums is a challenging and often quite difficult object of study. 
In this paper we investigate the finite sum of cosecants $\,\sum\csc\big(\varphi+a\pi l/n\big)\,,$ 
where the summation index $l$ runs through 1 to $n-1$ and $\varphi$ and $a$ are arbitrary parameters,
as well as several closely related sums, 
such as similar sums of a series of secants, of tangents and of cotangents.
These trigonometric sums appear in various problems in mathematics,
physics, and a variety of related disciplines.
Their particular cases were 
fragmentarily considered in previous works, 
and it was noted that even a simple particular case $\sum\csc\big(\pi l/n\big)$
does not have a closed--form, i.e.~a compact summation formula
(similar sums of sines, of cosines, of tangents, of cotangents and even of secants, if they exist,
possess such expressions).
In this paper, we derive several alternative representations for the above--mentioned sums,  
study their properties, relate them to many other finite and infinite sums, 
obtain their complete asymptotic expansions for large $n$
and provide accurate upper and lower bounds (e.g.~the typical relative error for the upper bound is lesser 
than $2\times10^{-9}$ for $n\geqslant10$ and 
lesser than $7\times10^{-14}$ for $n\geqslant50$, which is much better than the bounds we could find in previous works).
Our researches reveal that these sums are deeply related to several special numbers and functions, especially to 
the digamma function (furthermore, as a by-product, we obtain several interesting summation formul\ae~for the digamma function). 
Asymptotical studies show that these sums may
have qualitatively different behaviour depending on the choice of $\varphi$ and $a$.
In particular, as $n$ increases some of them may become sporadically large and we identify the terms of
the asymptotics responsible for such a behaviour. 
Finally, we also provide several historical remarks related to various sums considered in the paper.
We show that some results in the field either were rediscovered several times or can easily be deduced 
from various known formul\ae, including some formul\ae~dating back to the XVIIIth century.
\end{abstract}


\end{frontmatter}

\section{Introduction}
\subsection{A short historical survey and the motivation of the work}
Finite trigonometric sums and products often appear in analysis, in number theory, combinatorics, discrete mathematics,
and in many other fields of mathematics.
Some of them have been known for a very long time; for instance, this product
\be\label{h23894hd}
\prod_{l=0}^{n-1} \left(a^2-2ax\cos\left(\!\frac{\varphi+\,2\pi l\,}{n}\!\right) +x^2\right) = \, a^{2n}-2a^nx^n\cos\varphi+
x^{2n}\,,
\qquad \quad n\in\mathbbm{N}\,,
\ee
or the cotangent summation theorems 
\be\label{kjcwibaz}
\sum_{l=0}^{n-1} \ctg\!\left(\!\varphi+\frac{\,\pi l\,}{n}\!\right) = \,n\ctg n\varphi \,,
\qquad n\in\mathbbm{N}\,,\quad \frac{\varphi}{\pi}\notin\mathbbm{Q}\,,
\ee
\be\label{kjcwibaz2}
\sum_{l=0}^{n-1} \ctg^2\!\left(\!\varphi+\frac{\,\pi l\,}{n}\!\right) = \,n^2\csc^2 n\varphi - n\,,
\qquad n\in\mathbbm{N}\,,\quad \frac{\varphi}{\pi}\notin\mathbbm{Q}\,,
\ee
can be found in Euler's works (see the historical remark on p.~\pageref{rem5d45sx} for more details).
Others results, such as
\begin{eqnarray}
&&\sum_{l=1}^{n-1} \tg^2\!\left(\!\frac{\,\pi l\,}{n}\!\right) = \,n(n-1)\,,
\qquad n=3,5,7,\ldots\,, \label{cv309j}\\[3mm]
&&\sum_{l=1}^{n-1} \ctg\frac{\,\pi l\,}{n}\cdot\sin\frac{\,2\pi lk\,}{n} = \,n-2k\,,
\qquad k=1,2,\ldots,n-1\,,\label{cv309j2}\\[3mm]
&& \sum_{l=0}^{n-1} \sec
\frac{\pi l}{2n} \cdot\cos\frac{(2k+1)l\pi}{2n}=(-1)^k \biggl( n-2
\biggl\lfloor \frac{k+1}{2} \biggr\rfloor \biggr) \,,\qquad k=0,1,\ldots, 2n-1\,, \label{cv309j3}
\end{eqnarray}
where $\lfloor\,\rfloor$ is the integer part,
or finite sums of a mixed type, e.g.
\begin{eqnarray}\label{834un3u84}
\sum_{l=1}^{n} \Psi \left(
\frac{l}{n} \right) \cdot\cas \frac{2\pi l\nu}{n} \,= 
\begin{cases}
\displaystyle  n\left(\ln2 - \frac{\pi}{2}\right) + n\ln\sin\frac{\pi \nu}{n} + \pi \nu\,, 
& \nu=1,2,\ldots,n-1\,, \\[4mm]
\displaystyle   -n\ln n -n\gamma \,, & \nu=n\,,
\end{cases}
\end{eqnarray}
where $n\in\mathbbm{N}, \,\cas \alpha\equiv\sin\alpha +\cos\alpha\,$ and
 $\Psi$ stands for the digamma function,  
appeared in later works, and for some of them it can be very hard to identify the author 
and to date them. For instance, first two formul\ae~\eqref{cv309j}--\eqref{cv309j2} are usually credited to Stern 
and Eisenstein respectively, see e.g.~\cite[pp.~358 \& 360]{berndt_04}, but in fact formula \eqref{cv309j} directly follows 
from Euler's results \eqref{kjcwibaz2}\footnote{Setting $\varphi=\tfrac12\pi$ into \eqref{kjcwibaz2} immediately
yields \eqref{cv309j}. See also the historical remark on p.~\pageref{rem5d45sx} hereafter.} and should, therefore, be credited rather to Euler.
Formula \eqref{cv309j3} was obtained as an auxiliary
result in \cite[p.~94]{iaroslav_06}, and so far we could not find it in previous works.
As to identity \eqref{834un3u84}, it follows from the orthogonality properties of the $\cas$-function
and Gauss' results on the digamma function, dating back to the beginning of the XIXth century \cite{gauss_02}.

The finite sums are also very important for the applications, and often appear in physics and in a variety of related disciplines.
For instance, in digital signal processing, information theory, 
telecommunications and in their applications an immense role is played by the discrete orthogonal transforms, such as
the DFT (Discrete Fourier Transform), the DCT (Discrete Cosine Transform) and the DHT (Discrete Hartley Transform).\footnote{Most modern digital media,
such as digital images (e.g.~JPEG, HEIF), digital video (e.g.~MPEG, H.26x), digital audio (e.g.~Dolby
Digital, MP3, AAC), digital television (SDTV, HDTV, VOD), digital radio (DRM, AAC+, DAB+), speech coding
(AAC-LD, Siren, Opus) use discrete orthogonal transforms to compress the data. Such a (often lossy) compression allows to considerably
decrease the bitrate and leads in turn to a smaller bandwidth during the transmission.}
For example, the DHT is a pair of these beautiful and purely symmetrical identities:
\be\label{jc3809hfc2}
\begin{array}{l}
\displaystyle
\DHT\big[h(l)\big]\,\equiv\,
H(\nu)\,=\, \frac{1}{\sqrt{n\,}}\!\sum_{l=1}^{n} \, h (l) \cas \frac{2\pi l \nu}{n}\, , \qquad \nu=1,2,\ldots, n\,;\\[5mm]
\displaystyle
\DHT^{-1}\big[H(\nu)\big]\,\equiv
h(l)\,=\, \frac{1}{\sqrt{n\,}}\!\sum_{\nu=1}^{n} H(\nu) \cas \frac{2\pi \nu l}{n}\, , \qquad l=1,2,\ldots, n\,.
\end{array}
\ee
$H(\nu)$ is called the \emph{(direct) discrete Hartley transform}, $h(l)$ is called 
the \emph{inverse discrete Hartley transform} (e.g.~identity \eqref{834un3u84} is, up to a constant,
a Hartley transform of the digamma function of arguments uniformly distributed on the unit interval, i.e.~$\,\sqrt{n}\cdot\DHT\big[\Psi(l/n)\big]$).\footnote{The continuous version of the Hartley transform
was proposed by Ralph Hartley as a version of the Fourier integral identity in August 1941 \cite[p.~144]{hartley_01}.
It is commonly accepted that the discrete version, the DHT, was proposed by Ronald Bracewell 
in May 1983 \cite{bracewell_02}, \cite[Chapt.~4]{bracewell_01}, 
\cite[Eqs.~4.7.3--4.7.4]{olejniczak_01}, \cite{jones_01}, though of course such kind of identities could be known much earlier.
Note also that various definitions of the DHT may exist in the literature. 
In particular, the coefficient $\sqrt{n}$ for both transforms may 
be written as $1$ for the direct transform and as $1/n$ for the inverse transform respectively, and vice-versa. Also, both summation intervals may 
be taken $[0,n-1]$, and more generally $[k,k+n-1]$, instead of $[1,n]$, where $k=0,1,2,\ldots$} 
From the mathematical point of view, theses transformations are simply finite trigonometric sums.

Although a huge quantity of finite summation formul\ae~can be found in standard tables of series, 
such as those of Hansen \cite{hansen_01}, Jolley \cite{jolley_01}, and Prudnikov et
al. \cite{prudnikov_en}, \cite{gradstein_en}, these formul\ae~still continue to attract the attention of mathematicians, 
see e.g.~\cite{berndt_04,cvijovic_01,beck_02,grabner_01,chen_06,cvijovic_02,bettin_01,raigorodskii_01,harshitha_01,allouche_04,allouche_05,berndt_05,berndt_06}.
Quite often, finite trigonometric sums do not have simple evaluations in a closed--form.
In such cases, it may be useful to have a suitable asymptotic formula, especially for the applications
(e.g., in practice, the behaviour of DFT/DHT/DCT is often investigated only when $n$ is large). 
However, even the latter may become a very hard problem.

Consider the following finite trigonometric sum
\be\label{984ycbn492}
\sum_{l=1}^{n-1}  \csc\frac{\,\pi l\,}{n}\,,\qquad n\in\mathbbm{N}\setminus\{1\}\,.
\ee
This sum arises in various problems in mathematics, physics, 
as well as a variety of related disciplines, such as, for example, astronomy or cryptography
\cite{watson_02,watson_03,williams_01,duncan_0,duncan_00,corput_01,cochrane_01,peral_01,cochrane_02,alzer_01,
chen_05,chen_06,pomerance_01,harkins_01,hargreaves_01,sacha_01,thiede_01,majic_01,tong_01}.
In a recent paper written by a team of Chinese researchers \cite[Sec.~1]{tong_01}, it has been remarked that although
there are many results on the sums like $\sum_{l=1}^{n-1} \csc^{2p}(\pi l/n)$, $p\in\mathbbm{N}$,\footnote{See 
e.g.~\cite[vol.~1,\S~4.4.6]{prudnikov_en}, \cite{raigorodskii_01}.}
we still do not have a closed--form expression for \eqref{984ycbn492}.\footnote{By a closed--form expression
for the finite sum we mean a compact summation formula with a limited number of terms, which does not depend on the
length of the sum.}
The same observation about this sum was earlier made by Chen \cite[p.~74]{chen_06}, 
and accordingly to the latter, by Borwein \cite[p.~81]{chen_06}. 
This is actually one of the most curious facts about this sum: unlike similar sums 
of sines, of cosines, of cotangents, of tangents and even of secants (last two sums exist only if $n$ is odd), 
which can be evaluated in a closed--form, sum \eqref{984ycbn492} does not possess such a formula.
Furthermore, a whole chapter in \cite[Chapt.~7 entitled ``On a Sum of Cosecants'']{chen_06}
is devoted to the finite sum \eqref{984ycbn492}
and to the difficulties associated to the derivation of its properties.

In the present paper, we investigate a more general sum of a series of cosecants
\be\label{984ycbn492v2}
S_n(\varphi,a)\,\equiv\sum_{l=1}^{n-1}  \csc\!\left(\!\varphi+\frac{\,a\pi l\,}{n}\!\right) \,,
\qquad n\in\mathbbm{N}\setminus\{1\}\,,\qquad 
\varphi+\frac{\,a\pi l\,}{n}\neq\pi k\,,\quad k\in\mathbbm{Z}\,,
\ee
$l=1,2,\ldots,n-1$, where $\varphi$ and $a$ are some parameters,\footnote{These parameters may be interpreted
as the initial phase/angle and the scaling factor respectively. Furthermore, for the sake of convenience
and without loss of generality, we suppose that $a>0$.}
as well as several closely related sums, such as, for example, 
\be\label{984ycbn492v3}
C_n(\varphi,a)\,\equiv\sum_{l=1}^{n-1}  \sec\!\left(\!\varphi+\frac{\,a\pi l\,}{n}\!\right) \,,\qquad n\in\mathbbm{N}\setminus\{1\}\,,\qquad 
\varphi+\frac{\,a\pi l\,}{n}\neq\pi (k+\nicefrac12)\,,\quad k\in\mathbbm{Z}\,,
\ee
or a finite sum of a series of (co)tangents.
We provide several series and integral representations for these sums, establish their basic properties, 
relate them to some other finite sums,
derive their complete asymptotical expansions and obtain very tight upper and lower bounds for them. 
We show that these sums are closely connected with various special numbers and functions,
especially with the digamma function, which is omnipresent in formul\ae~related to these sums. 
Furthermore, some of these results may also be written as the summation formul\ae~for the digamma function;
for instance
\begin{eqnarray}
&&\displaystyle 
\sum_{l=1}^{n} \Psi\!\left(\frac{l}{2n}\right)\,= \, -n\big(\gamma+\ln2n\big) -\ln2 -
\frac{\,\pi\,}{2}\!\sum_{l=1}^{n-1}  \csc\frac{\,\pi l\,}{n}\label{98y3d}\\[6mm]
&&\displaystyle 
\sum_{l=0}^{n-1} \frac{2l+1}{2n}\cdot\Psi\!\left(\frac{2l+1}{2n}\right) = \,-\frac{n(\gamma+\ln4n)}{2}+\frac{\pi}{4}
\sum_{l=1}^{n-1}  \csc\frac{\,\pi l\,}{n}\,,\label{87tnvf984}\\[6mm]
&&\displaystyle \notag
\sum_{l=1}^{n-1} \left\{\Psi\!\left(\frac{l}{2n}+z\right) 
+ \Psi\!\left(\frac{l}{2n}-z\right) \! \right\}=\,2n\Psi\big(2nz\big) - 2\Psi\big(2z\big) 
-2n\ln 2n \, +2\ln2 - \\[5mm]
&&\qquad\qquad\qquad\qquad\qquad\qquad\qquad\qquad\qquad\displaystyle \label{93hx2y3bx}
- \pi\!\sum_{l=1}^{n-1}  \tg\!\left(\!\pi z+\frac{\,\pi l\,}{2n}\!\right)
\,,\qquad z\notin\mathbbm{Q}\,,
\end{eqnarray}
where the latter identity may equivalently be written in terms of $S_n(2\pi z,1)$
\be\label{982yxb7c}
\begin{array}{ll}
&\displaystyle
\sum_{l=1}^{n} \left\{\Psi\!\left(\frac{l}{2n}+z\right) 
+ \Psi\!\left(\frac{l}{2n}-z\right) \! \right\}
=\,2n\Psi\big(2nz\big) - 2\Psi\big(2z\big) +2\Psi\big(\nicefrac12+z\big)
-2n\ln n \, - 
 \\[8mm]
&\displaystyle \quad
-2(n-1)\ln2+\pi n\ctg(2\pi n z) - \pi\csc(2\pi z) \,- \pi\!
\sum_{l=1}^{n-1}  \csc\!\left(2\pi z+\frac{\,\pi l\,}{n}\!\right)
\,,\qquad z\notin\mathbbm{Q}\,,
\end{array}
\ee
see Theorem \ref{732xdyb3} and Corollary \ref{9837dyi2hd1}.\footnote{Formul\ae~\eqref{98y3d}, \eqref{93hx2y3bx}
and \eqref{982yxb7c} are the direct consequences of the results obtained in Corollary \ref{9837dyi2hd1}.
In particular, \eqref{98y3d} follows directly from the second identity of Corollary \ref{9837dyi2hd1};
\eqref{982yxb7c} is a rewritten version of the first identity of the same Corollary. As to \eqref{93hx2y3bx},
it is obtained from \eqref{982yxb7c} with the aid of the cotangent summation theorem \eqref{kjcwibaz} and the 
identities $\,\csc \alpha - \ctg\alpha = \tg(\alpha/2)\,$ and $\,\Psi(\nicefrac12)=-\gamma-2\ln2.$
Formula \eqref{87tnvf984} immediately follows from Corollary \ref{732xdyb3}.
These summation formul\ae~do not seem to be (widely) known, albeit they can be, of course, derived by other means.}\up{,}\footnote{The 
above relationships complete the summation formul\ae~obtained earlier in \cite[Appendix B, Eqs.~(B.6)--(B.11)]{iaroslav_07}.}
We also perform detailed asymptotical studies of $S_n(\varphi,a)$ and $C_n(\varphi,a)$.
They reveal that at large $n$ both sums may have qualitatively different behaviour depending on how the parameters $\varphi$ and $a$ are chosen. 
In particular, they have different leading terms, some of which may become sporadically large as $n$ increases. 
We finally compare our bounds and asymptotics to those obtained by other authors, and also provide
several important historical remarks, first of which comes below.\\

\noindent\textbf{Historical remark.}\label{rem5d45sx}
Formul\ae~\eqref{h23894hd}--\eqref{kjcwibaz2} appeared in one of Euler's fundamental works 
\emph{``Introduction to analysis of the infinite''} \cite{euler_04}, published for the first time in 1747 in Latin.
This book was later translated into several languages, including French in 1785, German in 1788, Russian in 1936 
and English in 1988.
Formula \eqref{h23894hd}, along with other formul\ae~of the same sort, 
were obtained in Chapter 9 of \cite[vol.~1]{euler_04}; 
formul\ae~\eqref{kjcwibaz}--\eqref{kjcwibaz2} were derived in Chapter 14 of \cite{euler_04}, see Fig.~\ref{97tvb098v}. 
\begin{figure}[!t]   
\centering
\includegraphics[width=0.728\textwidth]{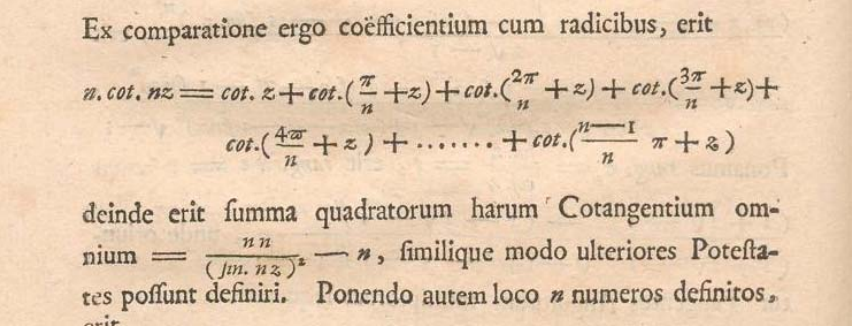}
\caption{A fragment of page 212 of \cite[Vol.~1, Chapt.~14]{euler_04}, where Euler provides
both cotangent summation theorems and states that similar relationships can be established for 
higher powers as well.}
\label{97tvb098v}
\end{figure}
It is interesting that nowadays these Euler's results, and more particularly that related to the sum of the squares of cotangents,
seem to be faded into oblivion. 
For example, the historical discussion in \cite[pp.~358 \& 369]{berndt_04} dates the latter type
of formul\ae~to the mid-XIXth century and, unfortunately, does not mention Euler. One may also mention Neville's rediscovery of \eqref{kjcwibaz2}
in \cite[p.~629]{neville_01}, as well as the interesting discussion in \cite[Remarks 2--3]{allouche_04}.
In addition to what was reported in the latter reference, it is worth noting that all the formul\ae~discussed in these two remarks 
are also the consequences of these forgotten results of Euler.\footnote{The first formula in Remark 2
of \cite{allouche_04} is obtained by letting $\varphi\to0$ in \eqref{kjcwibaz2}, which may also be written as \eqref{89ybwetv3};
first two formul\ae~in Remark 3 of \cite{allouche_04} are obtained by setting $\varphi=\frac12\pi/n$ in \eqref{kjcwibaz2} 
or in \eqref{89ybwetv3}.}

Besides, it should be noted that some finite sums with cosecants and secants were also known to Leonhard Euler.
For example, the formul\ae
\be\label{972ryhxded}
\sum_{l=0}^{n-1} (-1)^{l}\csc\!\left(\!\varphi+\frac{\,\pi l\,}{n}\!\right) = \,n\csc n\varphi\,, 
\qquad n=1,3,5,\ldots\,,\qquad \varphi\neq\pi(k- \nicefrac{l}{n})\,,
\ee
and 
\be\label{972ryhxded2}
\sum_{l=0}^{n-1} (-1)^{l}\sec\!\left(\!\varphi+\frac{\,\pi l\,}{n}\!\right) = \,n\sec n\varphi\,, 
\qquad n=1,3,5,\ldots\,,\qquad \varphi\neq\pi(k+\nicefrac12- \nicefrac{l}{n})\,,
\ee
$l=1,\ldots,n$, $k\in\mathbbm{Z}$,
appear in a slightly disguised form in \cite{euler_04}, see Fig.~\ref{ugf7u5rc6e}.
Furthermore, from the cotangent summation theorem \eqref{kjcwibaz}, it immediately follows that
\begin{eqnarray}
&&\displaystyle
\sum_{l=0}^{n-1} \tg\!\left(\!\varphi+\frac{\,\pi l\,}{n}\!\right) = \,
\begin{cases}
\,n\tg n\varphi  \,,& n=1,3,5,\ldots\\[1mm]
\,-n\ctg n\varphi\,,& n=2,4,6,\ldots
\end{cases}
\qquad \varphi\neq\pi(k+\nicefrac12- \nicefrac{l}{n})\\[3mm]
&&\displaystyle
\sum_{l=0}^{n-1} \csc\!\left(\!\varphi+\frac{\,2\pi l\,}{n}\!\right) = \,
\begin{cases}
\,n\csc n\varphi  \,,& n=1,3,5,\ldots\\[1mm]
\,0 \,,& n=2,4,6,\ldots
\end{cases}
\qquad \varphi\neq\pi(k- \nicefrac{2l}{n})
\end{eqnarray}
since $\,\tg \alpha=-\ctg(\alpha+\pi/2)\,$ and $\,\csc2\alpha=\tfrac12\big(\tg \alpha+\ctg \alpha\big)\,$. 
But $\,\csc(\varphi+\pi/2)=\sec\varphi\,$, hence the latter is also
\begin{eqnarray}
&&\displaystyle
\sum_{l=0}^{n-1} \sec\!\left(\!\varphi+\frac{\,2\pi l\,}{n}\!\right) = \,
\begin{cases}
\,n\sec n\varphi  \,,& n=1,3,5,\ldots\\[1mm]
\,0 \,,& n=2,4,6,\ldots
\end{cases}
\qquad \varphi\neq\pi(k+\nicefrac12- \nicefrac{2l}{n})
\end{eqnarray}
Therefore, at least one particular case of the summation formul\ae~and one
functional relationship for $S_n(\varphi,a)$,
see \eqref{97fsrc3y4} and \eqref{97vr4c3rcd} hereafter, were known as far back as Euler's time.
At the same time, the study of other particular cases of these sums seem to represent much more difficulties.
For instance, Watson \cite{watson_02,watson_03}, Hargreaves \cite{hargreaves_01}, Wilson \cite{williams_01}, Chen \cite[Chapt.~7]{chen_06}
and some other authors studied in detail the particular case $S_n(0,1)$, our formula \eqref{984ycbn492}.
They succeded in obtaining some of its properties, but did not find a closed--form formula for it. 
It is also interesting that at large $n$ the sum $S_n(0,1)$
behaves as $O(n\ln n)$, see Theorem \ref{ordtj56jx2} hereafter, while  
particular cases studied by Euler do not contain any trace of the logarithm. This 
indicates that $S_n(\varphi,a)$ may have, probably, a qualitatively different behaviour 
depending on the choice of $\varphi$ and $a$ (we come to confirm this guess in Theorem \ref{iu2389ghi}).

We also note that some summations formul\ae~for the finite sums of the type $\,\sum\csc^{2p}(\varphi+\pi l/n)\,$,
$\,\sum \sec^{2p}(\varphi+\pi l/n)\,$ and $\,\sum \tg^{2p}(\varphi+\pi l/n)\,$, 
$p\in\mathbbm{N}$ were also known to Euler. Indeed, since $\,\ctg^2\alpha=\csc^2\alpha -1\,$, Euler's formula \eqref{kjcwibaz2}
may also be written as
\be\label{89ybwetv3}
\sum_{l=0}^{n-1} \csc^2\!\left(\!\varphi+\frac{\,\pi l\,}{n}\!\right) = \,n^2\csc^2 n\varphi\,,
\qquad n\in\mathbbm{N}\,,\quad \varphi\neq\pi(k- \nicefrac{l}{n})\,,
\ee
for all $l=0,1,\ldots,n-1$ and $k\in\mathbbm{Z}$. The latter is also equivalent to
\be
\sum_{l=0}^{n-1} \sec^2\!\left(\!\varphi+\frac{\,\pi l\,}{n}\!\right) = 
\begin{cases}
\,n^2\sec^2 n\varphi\,,& n=1,3,5,\ldots\\[1mm]
\,n^2\csc^2 n\varphi\,,& n=2,4,6,\ldots
\end{cases}
\qquad \varphi\neq\pi(k+\nicefrac12- \nicefrac{l}{n})\,, 
\ee
with same $l$ and $k$ as above. Whence analogously 
\be\label{kj855}
\sum_{l=0}^{n-1} \tg^2\!\left(\!\varphi+\frac{\,\pi l\,}{n}\!\right) = 
\begin{cases}
\,n^2\sec^2 n\varphi - n \,,& n=1,3,5,\ldots\\[1mm]
\,n^2\csc^2 n\varphi - n \,,& n=2,4,6,\ldots
\end{cases}
\qquad \varphi\neq\pi(k+\nicefrac12- \nicefrac{l}{n})\,,
\ee

\begin{figure}[!t]   
\centering
\includegraphics[width=0.75\textwidth]{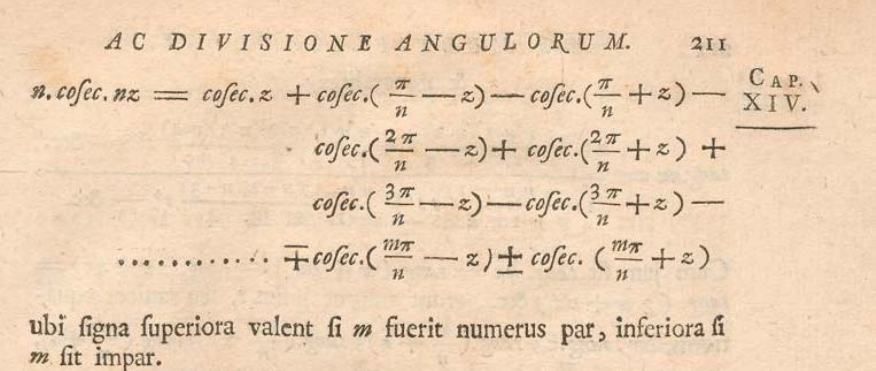}
\caption{A fragment of page 211 of \cite[Vol.~1, Chapt.~14]{euler_04}, where Euler evaluates
the alternating finite sum of cosecants, equivalent to our \eqref{972ryhxded}.\protect\footnotemark}
\label{ugf7u5rc6e}
\end{figure}
\footnotetext{In our usual notation, this Euler's formula reads
\be\label{984fhf}
(2m+1)\csc (2m+1)\varphi\,= \,\csc\varphi + \sum_{l=1}^{m} (-1)^l\!\left\{\csc\!\left(\!\frac{\,\pi l\,}{2m+1}+\varphi\!\right) 
- \csc\!\left(\!\frac{\,\pi l\,}{2m+1} - \varphi\!\right) \!\right\}\,.
\ee
But $\,\csc\big(\pi l/(2m+1)-\varphi\big)=\csc\big(\varphi +\pi -\pi l/(2m+1)\big)\,$, 
whence 
$$-\sum_{l=1}^{m} (-1)^l \csc\!\left(\!\frac{\,\pi l\,}{2m+1}-\varphi\!\right)  = 
\sum_{k=m+1}^{2m} (-1)^k \csc\!\left(\!\varphi+\frac{\,\pi k\,}{2m+1}\!\right)= 
\sum_{l=m+1}^{2m+1} (-1)^l\csc\!\left(\!\varphi+\frac{\,\pi l\,}{2m+1}\!\right) - \csc\varphi.$$
Substituting this expression into \eqref{984fhf} and setting $n=2m+1$ immediately yields \eqref{972ryhxded} (note 
that the $n$th term of this sum equals its zeroth term). 
By following the same line of argument, we may rewrite Euler's results from \cite[Vol.~1, Chapt.~14, p.~210]{euler_04}
as \eqref{972ryhxded2}.
}

Euler's formula \eqref{h23894hd}, as well as its variants and modifications, were also the subject of numerous rediscoveries. 
Without having the possibility to cite them all here,
we simply mention that, for example, in the above--cited book \cite{chen_06}, 
in Section 5, we find a lot of formul\ae~like \eqref{h23894hd}, but the name of Euler is not mentioned.
Formula (4) from \cite{chen_04} (also implicitly present in \cite[Sect.~5.3]{chen_06}) 
is actually a particular case of Euler's formula \eqref{h23894hd}.\footnote{Putting in \eqref{h23894hd} $\varphi=\alpha n$ and $a=1$ 
yields (4) from \cite{chen_04}. By the way, there is also a misprint in the right--hand side 
of this formula in \cite{chen_04}: ``$\cos(\alpha+2\pi k n)$'' should read ``$\cos(\alpha+2\pi k/ n)$''. This Euler's formula
is a very known result and is present in most tables of sums, see e.g.~entries \no 1.394 in \cite{gradstein_en}
and \no 6.1.2-7 in \cite[vol.~1, Chapt.~6]{prudnikov_en} .}
Furthermore, the main result of \cite[Eq.~1]{chen_04}, stated in the title of the paper as ``a new trigonometric identity''
(result also reproduced in \cite[Section 5.3, Eq.~5.22]{chen_06}) is, up to several coefficients, 
the partial logarithmic derivative of Euler's identity \eqref{h23894hd}.
Indeed, one can readily verify that
\be\notag
\left.\frac{\partial}{\partial x}\ln\left[x^{-n}\prod_{l=0}^{n-1} 
\left(a^2-2ax\cos\left(\!\frac{\varphi+\,2\pi l\,}{n}\!\right) +x^2\right)\right]\right|_{\substack{\varphi=0\\a=1}} 
= \, \sum_{l=0}^{n-1} \frac{x-x^{-1}}{1-2x\cos(2\pi l/n) +x^2}
\ee
and that
\be\notag
\left.\frac{\partial}{\partial x}\ln \frac{a^{2n}-2a^nx^n\cos\varphi+x^{2n}}{x^n}\right|_{\substack{\varphi=0\\a=1}} 
\,=\,\frac{n(x^n+1)}{x(x^n-1)}\,.
\ee
Evaluating the same derivative at $\varphi=\alpha n$ immediately yields \cite[Eq.~7]{chen_04}, 
as well as \cite[Eq.~5.20]{chen_06}.
Furthermore, proceeding analogously with similar products we
may also obtain these summation formul\ae:
\begin{eqnarray}
&&\displaystyle
\sum_{l=0}^{n-1} \frac{1}{\,1+2x\cos(2\pi l/n) +x^2\,} \,=\,
\begin{cases}
\displaystyle\frac{n(x^n-1)}{(x^2-1)(x^n+1)}\quad & n=1,3,5,\ldots\\[5mm]
\displaystyle\frac{n(x^n+1)}{(x^2-1)(x^n-1)} \quad & n=2,4,6,\ldots
\end{cases}
\qquad x\neq 1\,, \label{896gffrf5}\\[3mm]
&&\displaystyle
\sum_{l=1}^{n-1} \frac{1}{\,1\pm2x\cos(\pi l/n) +x^2\,} \,=\,\frac{n(x^{2n}+1)}{\,(x^2-1)(x^{2n}-1)\,} - \frac{x^{2}+1}{\,(x^2-1)^2\,}
\,,\qquad x\neq 1\,, \\[3mm]
&&\displaystyle
\sum_{l=0}^{n-1} \frac{1}{\,1-2x\cos\big(\pi\big(l+\frac12\big)/n\big) +x^2\,} \,=\,\frac{n(x^{2n}-1)}{\,(x^2-1)(x^{2n}+1)\,} 
\,,\qquad x\neq 1\,, \\[3mm]
&&\displaystyle
\sum_{l=1}^{n} \frac{1}{\,1\pm2x\cos\big(\pi l/\big(n+\frac12\big)\big) +x^2\,} \,
=\,\frac{\,n(x^{2n+1}\mp1)+x^{2n+1}\,}{\,(x^2-1)(x^{2n+1}\pm1)\,} - \frac{x}{\,(x^2-1)(x\pm1)\,}
\,,\qquad x\neq 1\,, \\[3mm]
&&\displaystyle
\sum_{l=0}^{n-1} \frac{1}{\,x+\ctg\big(\varphi+\pi l/n\big)\,} \,=
\,n\,\frac{\,(i-\ctg\varphi n)(ix+1)^{n-1}+(i+\ctg\varphi n)(ix-1)^{n-1}\,}
{\,(1+i\ctg\varphi n)(ix+1)^n+(1-i\ctg\varphi n)(ix-1)^n\,} \,, \label{987d3i} \\[3mm]
&&\displaystyle
\sum_{l=0}^{n-1} \frac{1}{\,x^2-1+\csc^2\!\big(\varphi+\pi l/n\big)  \,} \,=\label{987d3ig55}\\[3mm] 
&&\displaystyle
\qquad \qquad \qquad \qquad
=\,n\,\frac{\,(x+1)^{2n-1}-2x(x^2-1)^{n-1}+(x+1)^{2n-1}+4x\big(x^2-1\big)^{n-1}\sin^2\!\varphi n}
{\, x\big((x+1)^n-(x-1)^n\big)^2+4x\big(x^2-1\big)^n\sin^2\!\varphi n\,} \,, \notag
\end{eqnarray}
which are not included in the most known tables of sums and mathematical handbooks.\footnote{Reference handbooks 
\cite{hansen_01}, \cite{jolley_01}, \cite{prudnikov_en} and \cite{gradstein_en} missed these formul\ae~and, 
in addition, do not contain the main result of \cite{chen_04}. We also note that the left--hand side of our \eqref{896gffrf5} is also present in \cite[p.~308]{chen_04},
but our right--hand side differs from that obtained by Chen (Chen did not simplify his expression,
which is valid only for odd $n$).}\up{,}\footnote{First four formul\ae~are obtrained from Euler's products, that we can find, for example,
in entry \no 1.396 in \cite{gradstein_en}. Note, by the way, that there is an error in
\no 1.396-3: in the right--hand side in the numerator ``$x^{2n}-1$'' should be replaced by ``$x^{2n}+1$'' (this error remains uncorrected 
even in the latest, entirely retyped, enlarged and revised 8th edition of this book). Last two formul\ae~are obtained by a polynomial factorization
of $\,e^{-i\varphi n} (x+1)^n - e^{+i\varphi n} (x-1)^n $.}

\subsection{Notation and conventions}\label{notation}
By definition, the set of natural numbers $\mathbbm{N}$ does not include zero. 
Various special numbers are denoted as follows:
$\,\gamma\equiv\lim\limits_{n\to\infty}\!\big(H_n-\ln n\big)=0.577\ldots\,$ is Euler's constant, $\,H_n\equiv1+\nicefrac{1}{2}+
\nicefrac{1}{3}+\ldots+\nicefrac{1}{n}\,$
stands for the $n$th harmonic number,
${B}_n$ denotes the $n$th Bernoulli number. In particular
${B}_0=+1$, ${B}_1=-\nicefrac{1}{2}$, ${B}_2=+\nicefrac{1}{6}$,
${B}_3=0$, ${B}_4=-\nicefrac{1}{30}$, ${B}_5=0$, ${B}_6=+\nicefrac{1}{42}$, ${B}_7=0$,
${B}_8=-\nicefrac{1}{30}$, ${B}_9=0$, ${B}_{10}=+\nicefrac{5}{66}$,
${B}_{11}=0$, ${B}_{12}=-\nicefrac{691}{2730},\ldots$\footnote{For 
further values and definitions, see \cite[Tab.~23.2, p.~810]{abramowitz_01}, \cite[p.~5]{krylov_01}
Note also that there exist slightly different definitions for the
Bernoulli numbers, see e.g.~\cite[p.~91]{hagen_01},
\cite[pp.~32, 71]{lindelof_01}, \cite{watson_02,williams_01}, 
or \cite[pp.~3--6]{arakawa_01}.}
We also use numerous abbreviations for the functions and series.
In particular, $\lfloor z\rfloor$ stands for the the integer part of $z$, $\operatorname{tg}z$ for the tangent of $z$,
$\operatorname{ctg}z$ for the cotangent of $z$, $\operatorname{ch}z$
for the hyperbolic cosine of $z$, $\operatorname{sh}z$ for the
hyperbolic sine of $z$,
${\operatorname{th}}z$ for the hyperbolic tangent of $z$ and $\operatorname{cas}z$ for the 
Hartley kernel $\,\cas z\equiv\sin z +\cos z\,$.\footnote{Most of these notations 
are old european notations borrowed directly from Latin. E.g.~``$\operatorname{ch}$'' stands for \emph{cosinus hyperbolicus},
``$\operatorname{sh}$'' stands for \emph{sinus hyperbolicus}, etc. The notation
``$\operatorname{cas}$'' is especially used in the literature devoted to the Hartley transform.}
Writings $\Gamma(s)$, $\Psi(s)$, $\Psi_n(s)$, $\eta(s)$ and $\zeta(s)$ denote the gamma function, the digamma (psi) function,
the polygamma function of order $n$, the eta-function and the Euler--Riemann zeta-function of argument $s$ respectively.\footnote{By the $\eta$-function
we mean  $\,\eta(s)=\sum\limits_{n=1}^\infty (-1)^{n-1} n^{-s}$ for $\Re{s}>0.$
Note also that $\,\eta(s)=\big(1-2^{1-s}\big)\zeta(s).$} 
In order to be consistent 
with the previous notation of Watson \cite{watson_02,watson_03}, of Williams \cite{williams_01} and of some other authors,
we denote by $S_n$ the sums of cosecants \eqref{984ycbn492},
by $S_n(\varphi,a)$ the sums of cosecants \eqref{984ycbn492v2}, i.e.~$S_n=S_n(0,1)$, and 
by $C_n(\varphi,a)$ the sums of secants \eqref{984ycbn492v3} respectively.
We also define the \emph{principal value} for the sums, denoted $\pv$ for brevity. Let $\,\mathbbm{L}\subset\mathbbm{Z}\,$ 
be a nonempty set; then
\be
\pv \sum_{l\in\mathbbm{L}} f(l) \,\equiv\,\sum_{l\in\mathbbm{K}} f(l)\,, 
\qquad \mathbbm{K}=\mathbbm{L}\setminus\mathbbm{D}\,,
\ee
where $\,\mathbbm{D}\subset\mathbbm{L}\,$ is a zero-measure set such that $\,\sum 1/|f(l)| =0\,$, $l\in\mathbbm{D}.$
Finally, by the absolute error between the quantity $A$
and its approximated value $B$, we mean $|A-B|$.
Other notations are standard.\looseness=-1

\section{On a finite sums of cosecants}
\subsection{Preliminary remarks and basic properties}\label{h09387rxhxdws}
First of all, we note that $S_n(\varphi,a)$ is analytic everywhere in the complex plane, except 
at points $\varphi=\pi (k- al/n)$, $k\in\mathbbm{Z}$, $\,l=1,2,\ldots,n-1\,$,
at which it has simple poles.\footnote{Since any finite sum of analytic functions is also analytic.} At $n\to\infty$, this function
ceases to be analytic and the sum of cosecants diverges. 
Furthermore, it is not difficult to see that $S_n(\varphi,a)$ has also the following basic properties:
\begin{eqnarray}
&\displaystyle S_n(\varphi+(2k-1)\pi,a)\,=\,-S_n(\varphi,a)\,,\qquad k\in\mathbbm{Z}\,,\\[4mm]
&\displaystyle S_n(\varphi+2k\pi,a)\,=\,+S_n(\varphi,a)\,, \qquad k\in\mathbbm{Z}\,, \\[4mm]
&\displaystyle S_n(\varphi,a+2kn)\,=\,S_n(\varphi,a)\,, \qquad k\in\mathbbm{Z}\,,\\[4mm]
&\displaystyle S_n(\varphi,-a)\,=\,-S_n(-\varphi,+a)\,,\\[4mm]
&\displaystyle S_n(-\varphi,1+2kn)=+S_n(\varphi,1)\,,\qquad k\in\mathbbm{Z}\,,\label{9783dybx782}\\[4mm]
&\displaystyle S_{2n}(\varphi,a)\,=\,S_n\big(\varphi,\tfrac12 a\big)+S_n\big(\varphi+\tfrac12a\pi, \tfrac12 a\big)
+\csc\big(\varphi+ \tfrac12 a\pi\big)\,,\\[4mm]
&\displaystyle S_{kn}(\varphi,a)\,=\,\sum_{l=0}^{k-1} 
S_n\!\left(\varphi+\frac{\,a\pi l\,}{k},\frac{\,a\,}{k}\right)+S_k\big(\varphi, a\big)\,,\qquad k\in\mathbbm{N}\,, \label{8943cb34y9}\\[6mm]
&\displaystyle S_{kn}(\varphi,a)\,=\,\sum_{l=0}^{n-1} 
S_k\!\left(\varphi+\frac{\,a\pi l\,}{n},\frac{\,a\,}{n}\right)+S_n\big(\varphi, a\big)\,,\qquad k\in\mathbbm{N}\,,\label{8943cb34y9v2}\\[4mm]
&\displaystyle S_{n+1}(\varphi,a)\,=\,S_n\!\left(\varphi,\frac{\,an\,}{n+1}\right) 
+ \,\csc\left(\varphi+\frac{\,a\pi n\,}{n+1}\right)  \,,\qquad n\neq-1\,,\label{3984ychn34y8}
\end{eqnarray}
\begin{eqnarray}
&\displaystyle S_n\!\left(\!\varphi+\frac{\pi}{2n+1},\frac{2n}{2n+1}\right) \,=\, S_n\!\left(\!\varphi,\frac{2n}{2n+1}\right)
- (2n+1)\csc(2n+1)\varphi + \csc\varphi - \notag \\[4mm]
&\displaystyle\qquad\qquad\qquad\qquad
-\csc\!\left(\!\varphi+\frac{\pi}{2n+1}\right) - \csc\!\left(\!\varphi- \frac{\pi}{2n+1}\right)
 \qquad n=0,1,2,\ldots\label{97vr4c3rcd}
\end{eqnarray}
which may be obtained without much difficulty from the definition of $S_n(\varphi,a)$.
For instance, the multiplication theorem \eqref{8943cb34y9} is obtained as follows:
\begin{eqnarray}
&&\displaystyle \notag
S_{kn}(\varphi,a)\,=\!\sum_{l=1}^{nk-1}\!\csc\!\left(\!\varphi + \frac{\pi a l}{kn}\right) =
\sum_{l=1}^{n-1}\!\csc\!\left(\!\varphi + \frac{\pi \left(a/k\right) l}{n}\right) + \csc\!\left(\!\varphi + \frac{\pi a}{k}\right) +\\[4mm]
&&\displaystyle\notag\quad
+\sum_{l=n+1}^{2n-1}\!\csc\!\left(\!\varphi + \frac{\pi \left(a/k\right) l}{n}\right) + \csc\!\left(\!\varphi + \frac{2\pi a}{k}\right) 
+\,\ldots\, +\sum_{l=(k-1)n+1}^{kn-1}\!\csc\!\left(\!\varphi + \frac{\pi \left(a/k\right) l}{n}\right) =\\[4mm]
&&\displaystyle\notag\quad
=\underbrace{\sum_{l=1}^{n-1}\!\csc\!\left(\!\varphi + \frac{\pi \left(a/k\right) l}{n}\right)}_{S_n\big(\varphi,\frac{a}{k}\big)} + \,\ldots\,+
\underbrace{\sum_{l=1}^{n-1}\!\csc\!\left(\!\varphi + \frac{a\pi (k-1)}{k}+ \frac{\pi \left(a/k\right) l}{n}\right)}_{S_n\big(\varphi+\frac{a\pi (k-1)}{k},\frac{a}{k}\big)}
+ \underbrace{\sum_{l=1}^{k-1}\csc\!\left(\!\varphi + \frac{\pi a l}{k}\right)}_{S_k(\varphi,a)} .
\end{eqnarray}
Since the product is commutative, we also have \eqref{8943cb34y9v2}.
The recurrence relationship in $\varphi$, formula \eqref{97vr4c3rcd} follows from Euler's formula \eqref{972ryhxded}.

Furthermore, some particular values of $S_n(\varphi,a)$ may be evaluated in a closed--form:
\begin{eqnarray}
&\displaystyle S_0(\varphi,0)=-\csc\varphi \,, \quad  S_1(\varphi,a)=0  \,, \quad S_2(\varphi,a)=\csc(\varphi+a\pi/2)\,, \label{8943nc43} \\[3mm]
&\displaystyle S_n(\varphi,2kn)\,=\,(n-1)\csc\varphi \,, \qquad k\in\mathbbm{Z}\,,\label{973xbn3y4}
\end{eqnarray}

\begin{eqnarray}
&\displaystyle S_n(\varphi,2+2kn)\,=
\begin{cases}
n\csc n\varphi - \csc\varphi \,,\qquad & n=1,3,5,\ldots\\
- \csc\varphi \,,\qquad & n=2,4,6,\ldots
\end{cases} \qquad k\in\mathbbm{Z}\,.\label{97fsrc3y4}
\end{eqnarray}
These formul\ae~are obtained as follows.
In order to get \eqref{8943nc43}, we use the the recurrence formula \eqref{3984ychn34y8}.
Setting $n=1$ into the latter, we have: $S_2(\varphi,a)=S_1(\varphi,a/2)+\csc(\varphi+a\pi/2)$. 
But by \eqref{984ycbn492v2} $S_2(\varphi,a)=\csc(\varphi+a\pi/2)$, whence $\,S_1(\varphi,a/2)=0\,$ and so does $\,S_1(\varphi,a).$
The same result also straightforwardly follows from Lemma \ref{mlemma} [see hereafter]. As to \eqref{973xbn3y4}, it directly follows from the 
definition of $S_n(\varphi,a)$, see \eqref{984ycbn492v2}. Finally, \eqref{97fsrc3y4} is the consequence of formul\ae~given in 
the historical remark on p.~\pageref{rem5d45sx}.

\begin{remark}[A representation related to Euler--like product]
The sum $S_n(\varphi,a)$ may also be formed from a polynomial 
\be
P_{n-1} (x;\varphi,a)\,\equiv\,\prod_{l=1}^{n-1} \left\{x+ \sin\!\left(\!\varphi+\frac{\,a\pi l\,}{n}\right)\!\right\}=\, x^{n-1}+
\sum_{l=1}^{n-2} T_{l,n}(\varphi,a) \, x^l 
\ee
where trigonometric coefficients $T_{l,n}(\varphi,a)$ are analytic functions of $\varphi$ and $a$.
In particular
\begin{eqnarray}
&&\displaystyle P_1(x;\varphi,a)\,=\,x+\sin\!\left(\!\varphi+\frac{\,a\pi\,}{2}\!\right) \,,\notag\\[3mm]
&&\displaystyle P_2(x;\varphi,a)\,=\,x^2+
x \left\{\sin\!\left(\!\varphi+\frac{\,a\pi\,}{3}\!\right)+\sin\!\left(\!\varphi+\frac{\,2a\pi\,}{3}\!\right)\!\right\}
+ \sin\!\left(\!\varphi+\frac{\,a\pi\,}{3}\!\right)\cdot\sin\!\left(\!\varphi+\frac{\,2a\pi\,}{3}\!\right)\,,\notag\\[3mm]
&&\displaystyle \ldots\ldots\ldots\notag
\end{eqnarray}
Hence, it is not difficult to see that
\be
\sum_{l=1}^{n-1}  \csc\!\left(\varphi+\frac{\,a\pi l\,}{n}\!\right)=\,\left.\frac{\partial}{\partial x}\ln P_{n-1} (x;\varphi,a)\right|_{\substack{x=0}} 
\ee
\end{remark}

We conclude this section with a graph of $S_n(\varphi,a),$ depicted in Fig.~\ref{g6hytbhfw}.
It is quite remarkable that this sum has a qualitatively different behaviour, depending on how
the initial phase $\varphi$ and the scaling factor $a$ are chosen. The explanation of this interesting phenomenon
is given in Section \ref{4c312ct3q4rgfwa} (and more particularly in Section \ref{08r7ch3489y}) 
and is due to the presence of several quite different in nature terms in the leading part
of the asymptotical expansion of $S_n(\varphi,a)$, whose presence and contribution
depend on $\varphi$ and $a$.

\begin{figure}[!t]   
\centering
\includegraphics[width=0.48\textwidth]{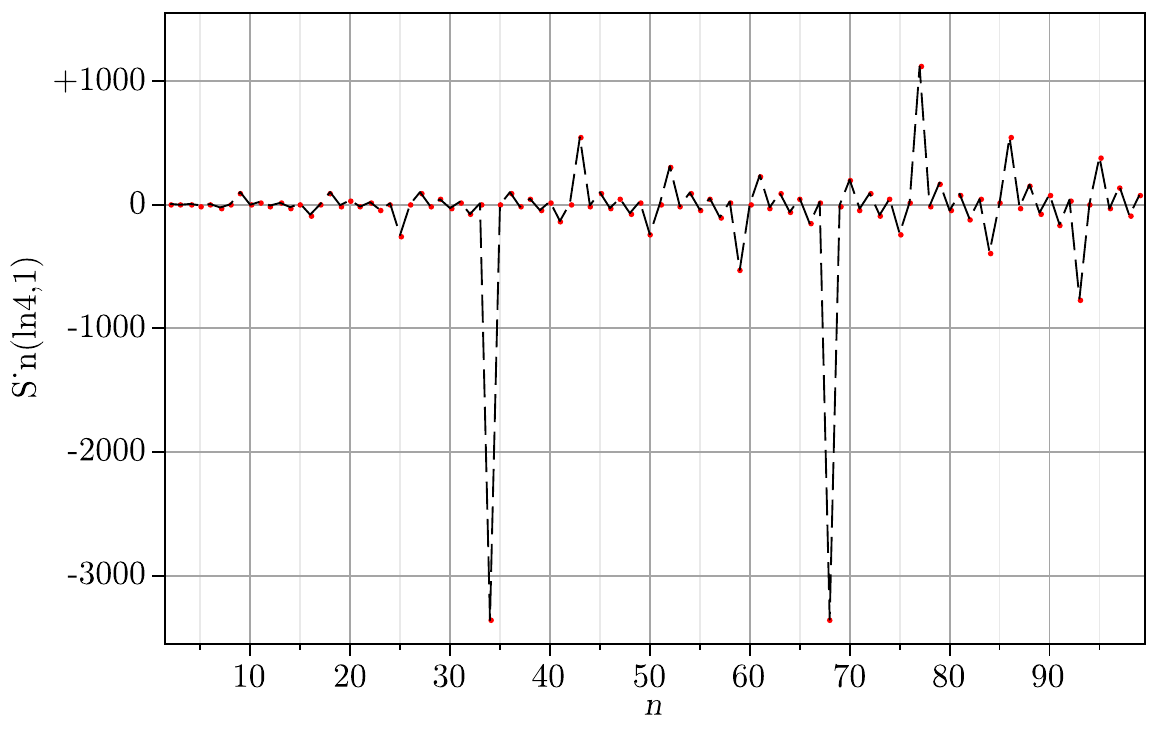}
\includegraphics[width=0.45\textwidth]{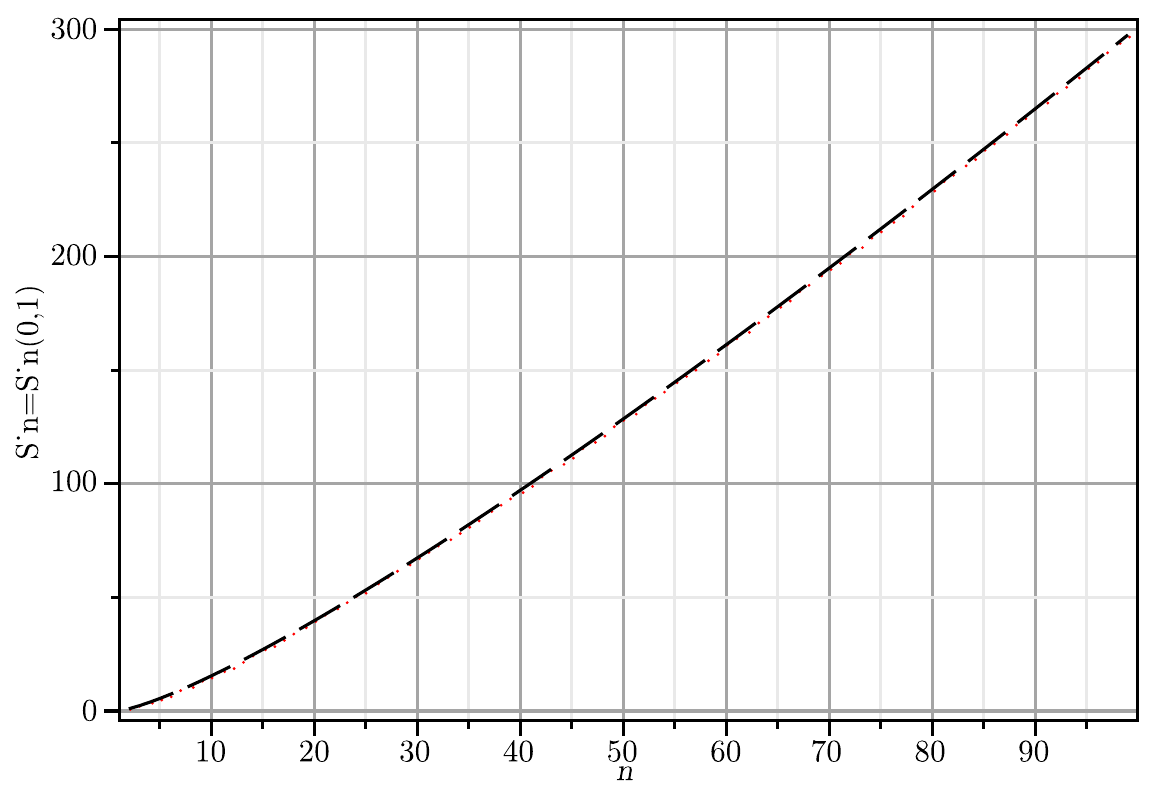}\\
\includegraphics[width=0.48\textwidth]{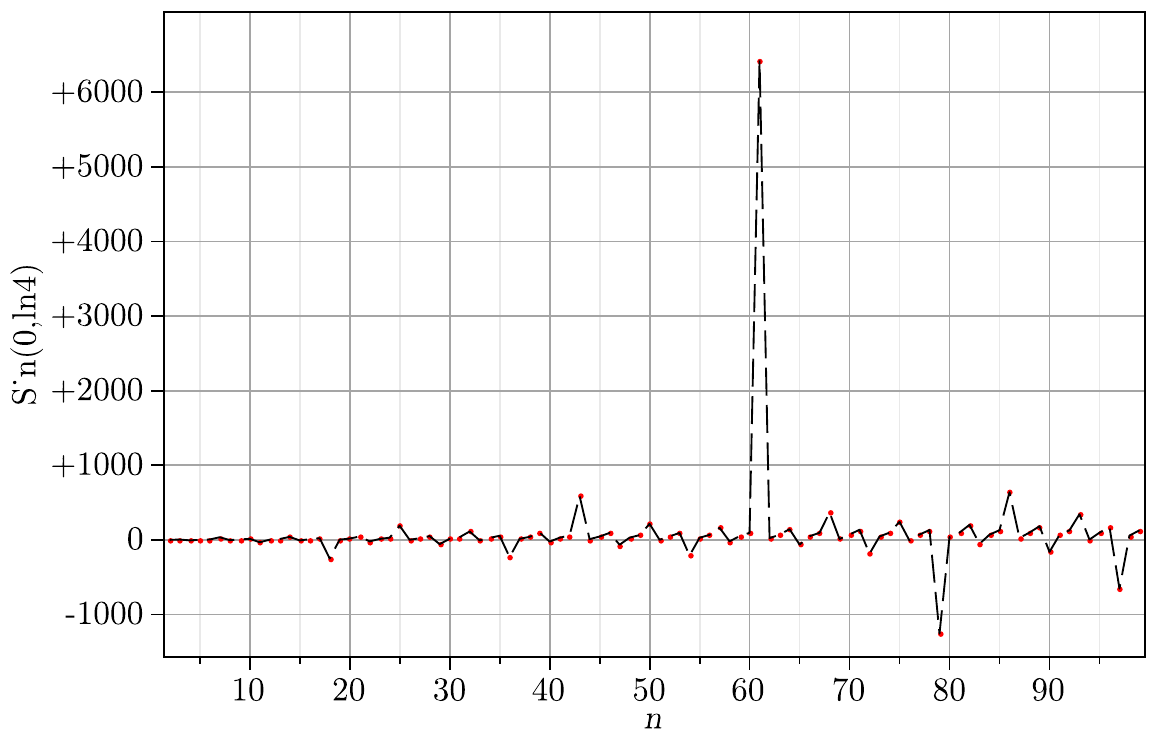}
\includegraphics[width=0.45\textwidth]{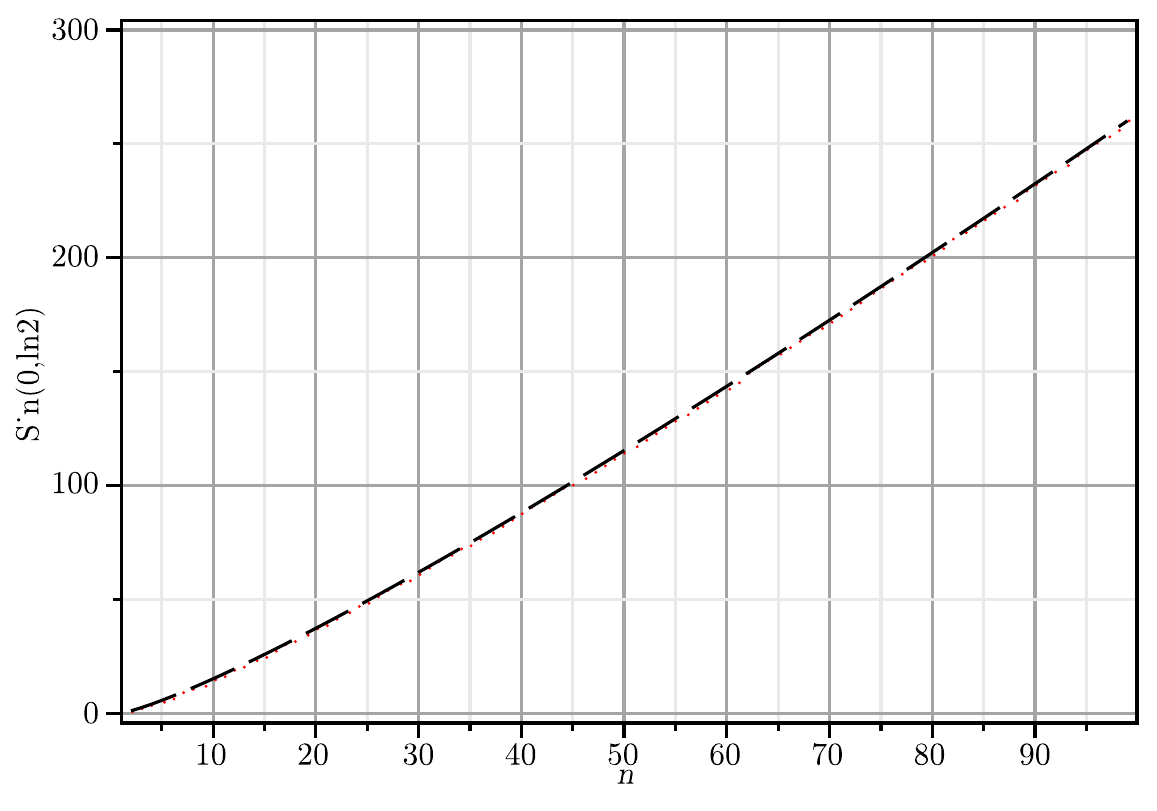}\\
\caption{The sum $S_n(\varphi,a)$ as a function of $n$ for various $\varphi$ and $a$.
Top: $S_n(\varphi,1)$ for $\varphi=2\ln2$ and $\varphi=0$; 
buttom: $S_n(0,a)$ for $a=2\ln2$ and $a=\ln2$.
The value of the sum is given by the red point, the dashed black line is shown only for better visualisation.
One may clearly observe very different behaviour of this sum, depending on the choice
of the parameters $\varphi$ and $a$. Note also
that $\ln{n}\protect\notin\mathbbm{Q}$ if $n\in\mathbbm{N}.$ 
}
\label{g6hytbhfw}
\end{figure}

\subsection{Integral representations}
Below, we derive an important integral representation for the sum $S_n(\varphi,a)$, which is useful for the
establishment of some of the properties of $S_n(\varphi,a)$.
\begin{lemma}[Improper integral representation]\label{mlemma}
The sum of cosecants \eqref{984ycbn492v2} may be represented via the following integral:
\be\notag
\sum_{l=1}^{n-1}  \csc\!\left(\varphi+\frac{\,a\pi l\,}{n}\!\right)= \,\frac{\,2n\,}{\pi}\!
\int\limits_0^\infty \!\frac{\,\sh\big[ax(n-1) \big] \,}{\sh ax \cdot\ch nx} \ch\!\left[nx\!\left(\frac{2\varphi}{\pi}+a-1\right)\right] dx\,,\qquad 
\ee
where $\,\displaystyle-\frac{a\pi}{n}<\Re\varphi <+\frac{a\pi}{n}+\pi(1-a)$\,.
In particular, for $S_n$ we have a particularly beautiful integral:
\be\notag
\sum_{l=1}^{n-1}  \csc\frac{\,\pi l\,}{n} \,= \,\frac{\,2n\,}{\pi}\!
\int\limits_0^\infty \!\left(\frac{\,\th nx\,}{\th x} - 1\right) dx\,.
\ee
\end{lemma}

\begin{proof}
Consider the following classic result due to Euler:
\be\label{hd2893dh2}
\int\limits_0^\infty \!\frac{\,x^{p-1}}{\,1+x\,} \, dx\,=\,\pi\csc p\pi\,,\qquad 0<\Re{p}<1 \,,
\ee
see e.g. \cite[\no 3.222-2]{gradstein_en}, \cite[\no 856.02]{dwigth_01_en}, \cite[p.~170 \& p.~172, \no 5.3.4.20-1]{mitrinovic_02}, 
\cite[p.~125, \no 878]{volkovyskii_01_eng}. 
Setting $\,p=\varphi/\pi +al/n\,$, $n\in\mathbbm{N}$, and summing the right-hand side from $l=1$ to $l=n-1$, we have
\be\notag 
\sum_{l=1}^{n-1}  \csc\!\left(\varphi+\frac{\,a\pi l\,}{n}\!\right) \, = \,\frac{\,1\,}{\pi}\!
\int\limits_0^\infty \! \!\frac{\,x^{\frac{\varphi}{\pi}}}{\,x(1+x)\,} \sum_{l=1}^{n-1} x^\frac{al}{n} \, dx\,
=\,\frac{\,1\,}{\pi}\! \int\limits_0^\infty \! \!\frac{\,x^{\frac{\varphi}{\pi}}\big(x^a-x^\frac{a}{n}\big)}
{\,x(1+x)\big(x^\frac{a}{n}-1\big)\,}  \, dx \,.
\ee
Making a change of variable $x=e^\frac{yn}{a}$, the latter integral becomes:
\be\notag
\int\limits_0^\infty \! \!\frac{\,x^{\frac{\varphi}{\pi}}\big(x^a-x^\frac{a}{n}\big)}{\,x(1+x)\big(x^\frac{a}{n}-1\big)\,}  \, dx\,
=\,\frac{\,n\,}{a}\!\! \int\limits_{-\infty}^{+\infty}\!\! \frac{e^\frac{y\varphi n}{\pi a}
\big(e^{yn}-e^y\big)}{\big(e^y-1\big)\big(1+e^\frac{yn}{a}\big)} \, dy\,.
\ee
Splitting it into two integrals
and making several suitable changes of variable yields after some algebra: 
\be\notag
\begin{array}{ll}
\displaystyle
\int\limits_{-\infty}^{+\infty}\!\! \frac{e^\frac{y\varphi n}{\pi a}\big(e^{yn}-e^y\big)}{\big(e^y-1\big)\big(1+e^\frac{yn}{a}\big)} \, dy \,
\; &\displaystyle =\,\int\limits_{-\infty}^{0} \!\!\ldots \, dy \, + \int\limits_{0}^{\infty}\!\!\ldots \, dy \, =\,\int\limits_{0}^{\infty}\! \frac{e^\frac{y\varphi n}{\pi a}\big(e^{yn}-e^y\big) 
- e^{1-\frac{y\varphi n}{\pi a}+\frac{yn}{a}}\big(e^{-yn}-e^{-y}\big) }
{\big(e^y-1\big)\big(1+e^\frac{yn}{a}\big)} \, dy   \\[8mm]
&\displaystyle =\,a\!\int\limits_{0}^{\infty}\! \left\{\sh\!\left[t\!\left(\frac{2\varphi n}{\pi}-n+2an-a\right) \right] 
- \sh\!\left[t\!\left(\frac{2\varphi n}{\pi}-n+a\right) \right] \!\right\}
\frac{dt}{\sh at \cdot \ch nt} 
\end{array}
\ee
\be\label{9348ucn93}
\begin{array}{ll}
\displaystyle
&\displaystyle =\,2a\!
\int\limits_0^\infty \!\frac{\,\sh\big[at(n-1) \big] \,}{\sh at \cdot\ch nt} \ch\!\left[nt\!\left(\frac{2\varphi}{\pi}+a-1\right)\right] dt\,,
\end{array}
\ee
where for the last change of variable we set $y=2at$. Besides, it is straightforward to see that if $a=1$, the latter integral becomes 
an even function of $\varphi$, whence we obtain property \eqref{9783dybx782}.

The domain of convergence stated in Lemma \ref{mlemma} follows from that of formula \eqref{hd2893dh2} and the fact that $1\leqslant l \leqslant n-1$. Indeed,
the condition $0<\Re{p}<1$ leads to $-\frac{a\pi l}{n}<\Re{\varphi}<\pi - \frac{a\pi l}{n}$, since $a$ was chosen positive. 
The latter inequality should hold for each $l$ ranging from 1 to $n-1;$ therefore, the domain of convergence is 
\be\notag
\bigcap_{l=1}^{n-1} \left(-\frac{a\pi l}{n}, \,\pi - \frac{a\pi l}{n}\right) 
= \left(-\frac{a\pi}{n}, \,\pi +\frac{a\pi}{n} -a\pi\right)\,,\qquad a>0\,, \quad n=2,3,4,\ldots
\ee

As to the case $\varphi=0$, $a=1$, it is sufficient to put these values into the last integral in \eqref{9348ucn93} and then 
remark that $\sh\big[(n-1)t\big]=\sh nt\ch t - \ch nt\sh t. $
\end{proof}

\subsection{Series representations}
In this section we obtain several alternative series representations for the function $S_n(\varphi,a)$, given 
as theorems.

\begin{lemma}[Series representation via a finite sum of cotangents]
The particular case $S_n(\varphi,1)$, $\varphi\neq0$, admits the following representation via a finite series of cotangents 
\be\notag
\sum_{l=1}^{n-1}  \csc\!\left(\!\varphi+\frac{\,\pi l\,}{n}\!\right) \,= 
\sum_{l=1}^{n-1}  \ctg\!\left(\!\frac{\varphi}{2}+\frac{\,\pi l\,}{2n}\!\right) - n\ctg n\varphi +\ctg\varphi\,.
\ee
For the special particular case $S_n$, we have a very simple formula
\be\label{0984cnyh8u}
\sum_{l=1}^{n-1}  \csc\frac{\,\pi l\,}{n}\,= \, \sum_{l=1}^{n-1}  \ctg\frac{\,\pi l\,}{2n}\,.
\ee
\end{lemma}

\begin{proof}
Using the elementary identity $\csc\alpha=\ctg\frac12\alpha-\ctg\alpha$ and the cotangent summation theorem \eqref{kjcwibaz},
we immediately obtain the main result of the Lemma. The particular case $S_n(0,1)$ is obtained by letting $\varphi\to0$. 
By noticing that
\be\notag
\lim_{\varphi\to0}\big(\ctg\varphi -  n\ctg n\varphi\big) \,= \lim_{\varphi\to0} \left\{\frac{\big(n^2-1\big)\varphi}{3}
+ O\big(\varphi^3\big) \! \right\} = \,0\,.
\ee
we arrive at \eqref{0984cnyh8u}.
\end{proof}

\begin{theorem}[Digamma finite series representation, particular case]\label{732xdyb3}
The sum $S_n$ is related to a finite sum involving the digamma function
\be\label{87knhoph984}
\sum_{l=1}^{n-1}  \csc\frac{\,\pi l\,}{n} = \,\frac{2n (\gamma+\ln4n)}{\pi}+
\frac{4}{\pi}\sum_{l=0}^{n-1} \frac{2l+1}{2n}\cdot\Psi\!\left(\frac{2l+1}{2n}\right)
\ee
\end{theorem}

\begin{proof}
Using this representation \cite[p.~561, Eq.~(54)]{iaroslav_07} for the cotangent, which may also be found in \cite[p.~98]{iaroslav_06}, 
the right--hand side of \eqref{0984cnyh8u} may be written as a double sum:
\be\label{nbxwioyew87}
\sum_{l=1}^{n-1}  \csc\frac{\,\pi l\,}{n}\,= \, \sum_{l=1}^{n-1}  \ctg\frac{\,\pi l\,}{2n}\,=
\,-\frac{1}{\,n\,} \sum_{r=1}^{2n-1}\sum_{l=1}^{n-1} r \sin\frac{\,\pi r l\,}{n}\,
\ee
But it is well--known that
\be\notag
\sum_{l=1}^{n-1} \sin\frac{\,\pi r l\,}{n}\, =
\begin{cases}
\ctg\dfrac{\,\pi r\,}{2n} \,,\qquad  	& r=1,3,5,\ldots\\[1mm]
0\,,			\qquad 			& r=2,4,6,\ldots
\end{cases}
\ee
Hence, \eqref{nbxwioyew87} reduces to
\be\label{97690hgfd}
\sum_{l=1}^{n-1}  \csc\frac{\,\pi l\,}{n}\,= \, -\frac1n\sum_{k=0}^{n-1}  (2k+1)\ctg\frac{\,\pi (2k+1)\,}{2n}
\ee
where we put $r=2k+1$.\footnote{Formula \eqref{97690hgfd} also appears (without proof) in \cite[p.~80]{chen_06}.}
Consider now the Gauss' Digamma theorem
\be
\Psi \biggl(\!\frac{r}{m} \vphantom{\frac{1}{2}} \!\biggr) =\,-\gamma-\ln2m-\frac{\pi}{2}\ctg\frac{\,\pi r\,}{m} 
+ \sum_{k=1}^{m-1} \cos\frac{\,2\pi r k\,}{m} \cdot\ln\sin\frac{\pi k}{m}\,,
\ee
where $r=1, 2,\ldots, m-1 $, $m=2,3,4,\ldots\,$, see e.g.~\cite[Eqs.~(B.4)]{iaroslav_07}.
First, set in this formula $r=2l+1$ and $m=2n$. 
Now, multiply both sides by $(2l+1)$ and sum the result over $l$ from $l=0$ to $l=n-1$.
This yields
\be\notag
\begin{array}{ll}
&\displaystyle 
\sum_{l=0}^{n-1} (2l+1) \Psi\!\left(\frac{2l+1}{2n}\right) = -(\gamma+\ln4n)\underbrace{\sum_{l=0}^{n-1} (2l+1)}_{n^2}
- \,\frac\pi2\underbrace{\sum_{l=0}^{n-1}  (2l+1)\ctg\frac{\,\pi (2l+1)\,}{2n}}_{\text{see \eqref{97690hgfd}}} + \\[10mm]
&\displaystyle \qquad\qquad
+ \sum_{k=1}^{n-1}\ln\sin\frac{\pi k}{n} \underbrace{\sum_{l=0}^{n-1} (2l+1) \cos\frac{\,\pi k (2l+1)\,}{n}}_{0} =
-n^2(\gamma+\ln4n) + \frac{\pi n}{2}\sum_{l=1}^{n-1}  \csc\frac{\,\pi l\,}{n}	
\end{array}
\ee
Dividing both sides by $2n$, we immediately obtain \eqref{87tnvf984},
which is equivalent to \eqref{87knhoph984}. 
\end{proof}

The connection, that we have just found between $S_n$ and the digamma function, is not accidental. 
Below, we come to show that there exist numerous relationships between $S_n(\varphi,a)$ and the digamma function.
In particular, $S_n(\varphi,a)$ may be expressed, in two very different ways, in terms of the digamma functions only.
These two cases are given as corresponding theorems.
It is also quite notable the manner in which the argument $n$ enters in both expressions: in the former
it contributes to the numerator of the argument of the digamma function, while in the latter 
it contributes to its denominator (it is similar in this regard to some Ramanujan's identities). 

\begin{theorem}[Digamma infinite series representation]\label{oiue2ynx}
The function $S_n(\varphi,a)$ may be expanded into the infinite series involving the digamma functions only:
\be\label{904cj9384nvc431u}
\begin{array}{ll}
\displaystyle 
\sum_{l=1}^{n-1}  \csc\!\left(\varphi+\frac{\,a\pi l\,}{n}\!\right) 
=\,\frac{\,n\,}{a\pi}\sum_{k=0}^{\infty} (-1)^k &\displaystyle \left\{  
\Psi\!\left(\frac{nk}{a}+n+\frac{n\varphi}{a\pi}\right) - \Psi\!\left(1+\frac{nk}{a}+\frac{n}{a}-\frac{n\varphi}{a\pi}-n\right) -  \right.\\[8mm]
&\displaystyle \quad\left.
- \Psi\!\left(1+\frac{nk}{a}+\frac{n\varphi}{a\pi}\right) + \Psi\!\left(\frac{nk}{a} +\frac{n}{a}-\frac{n\varphi}{a\pi}\right) \! \right\}\,,
\end{array}
\ee
where the parameters $\varphi$ and $a$ are chosen as stated in \eqref{984ycbn492v2}.
For $\varphi=0$ and $a=1$, i.e. for $S_n$, we have a qualitatively different expression:
\be\notag
\sum_{l=1}^{n-1}  \csc\frac{\,\pi l\,}{n}\,= \,\frac{\,2nH_{n}\,}{\pi}
\,-\,\frac{\,2(1-\ln2)\,}{\pi}\,+\,\frac{\,2n\,}{\pi}\!\sum_{k=1}^{\infty} (-1)^k \Big\{\Psi(nk+n) -\Psi(nk) \Big\}\,,
\ee
where for large $n$ the general term of the series $\Psi(nk+n) -\Psi(nk) \sim 1/k\,$ as $k\to\infty.$
\end{theorem}

\begin{proof}
By expanding the hyperbolic secant into the uniformly convergent geometric series
\be
\frac{1}{\,\ch nx\,}  \,=\,2\!
\sum_{k=0}^{\infty} (-1)^k e^{-n(2k+1)x}\,, \qquad \text{for }\, \Re(nx)>0\,,
\ee
and by using 
\be
\int\limits_0^\infty \! e^{-\alpha x} \, \frac{\sh\beta x}{\sh bx}\,dx \,=\, \frac{1}{2b}\left\{\Psi\left(\frac12+\frac{\alpha+\beta}{2b}\right)-
\Psi\left(\frac12+\frac{\alpha-\beta}{2b}\right)\!\right\}\,, \qquad \Re(\alpha+b - \beta)>0\,,
\ee
see e.g.~\cite[\no 3.541-2]{gradstein_en}, \cite[Vol.~I, Sec.~1.7.2, Eqs.~(14)--(15)]{bateman_01},
we see at once that 
\be\notag
\begin{array}{ll}
\displaystyle 
\int\limits_{0}^{\infty}\! \sh\!\left[x\!\left(\frac{2\varphi n}{\pi}-n+2an-a\right) \right] &\displaystyle
\frac{dx}{\sh ax \cdot \ch nx} \,=\\[7mm]
&\displaystyle\!\!\!\!\!
= \,2\!\sum_{k=0}^{\infty} (-1)^k \!\!
\int\limits_{0}^{\infty}\! e^{-n(2k+1)x}  \, \frac{\sh\!\left[x\!\left(\frac{2\varphi n}{\pi}-n+2an-a\right) \right] }
{\sh ax} \, dx \\[7mm]
&\displaystyle\!\!\!\!\! = \,\frac{1}{\,a\,}\!\sum_{k=0}^{\infty} (-1)^k \!
\left\{  \!
\Psi\!\left(\frac{nk}{a}+n+\frac{n\varphi}{a\pi}\right) - \Psi\!\left(1+\frac{nk}{a}+\frac{n}{a}-\frac{n\varphi}{a\pi}-n\right) \!\right\}
\end{array}
\ee
Similarly
\be\notag
\begin{array}{ll}
\displaystyle 
\int\limits_{0}^{\infty}\! \sh\!\left[x\!\left(\frac{2\varphi n}{\pi}-n+a\right) \right] &\displaystyle
\frac{dx}{\sh ax \cdot \ch nx} \,=\\[7mm]
&\displaystyle\!\!\!\!\!
= \,2\!\sum_{k=0}^{\infty} (-1)^k \!\!
\int\limits_{0}^{\infty}\! e^{-n(2k+1)x}  \, \frac{\sh\!\left[x\!\left(\frac{2\varphi n}{\pi}-n+a\right) \right] }
{\sh ax} \, dx \\[7mm]
&\displaystyle\!\!\!\!\! = \,\frac{1}{\,a\,}\!\sum_{k=0}^{\infty} (-1)^k \!
\left\{  \!
\Psi\!\left(1+\frac{nk}{a}+\frac{n\varphi}{a\pi}\right) - \Psi\!\left(\frac{nk}{a} +\frac{n}{a}-\frac{n\varphi}{a\pi}\right) \!\right\}
\end{array}
\ee
Substituting these two results into Lemma \ref{mlemma} and using formula \eqref{9348ucn93}, we immediately obtain the first expansion stated in Theorem \ref{oiue2ynx}.

For the case $\varphi=0$ and $a=1$, 
it is more simple to directly start from the second formula of Lemma \ref{mlemma}.
Following the same line of reasoning as above, we have
\be\label{u9h23dfe}
\begin{array}{ll}
\displaystyle 
\int\limits_0^\infty \!\left(\frac{\,\th nx\,}{\th x} - 1\right)   dx\; & \displaystyle=
\int\limits_0^\infty \!\frac{\,\sh\big[(n-1)x\big]\,}{\sh x \cdot \ch nx}\, dx\,  \\[6mm]
&\displaystyle  
=\,2\!\sum_{k=0}^{\infty} (-1)^k \!\int\limits_0^\infty \! e^{-n(2k+1)x}
\frac{\,\sh\big[(n-1)x\big]\,}{\sh x}\, dx\\[6mm]
&\displaystyle  =
\sum_{k=0}^{\infty} (-1)^k \Big\{\Psi(nk+n) -\Psi(1+nk) \Big\} \\[6mm]
&\displaystyle  =  
\Psi(n) -\Psi(1) + 
\sum_{k=1}^{\infty} (-1)^k\left\{\Psi(nk+n) -\Psi(nk) - \frac{1}{\,nk\,}\right\}\\[6mm]
&\displaystyle 
=  \, H_{n-1} + \frac{\,\ln2\,}{n} + 
\sum_{k=1}^{\infty} (-1)^k\Big\{\Psi(nk+n) -\Psi(nk) \Big\}\,,
\end{array}
\ee
since $\,\Psi(1+z)=\Psi(z)+1/z\,$, $\,\Psi(n) - \Psi(1) = \Psi(n)+\gamma = H_{n-1}\,$
and where we also used the Mercator series for $\ln2.$ Multiplying \eqref{u9h23dfe} by $2n/\pi$ and accounting for the fact that 
$H_n=H_{n-1}+1/n$ yields the second formula of Theorem \ref{oiue2ynx}.
\end{proof}

\begin{theorem}[Digamma finite series representations]\label{oiue2ynx2}
The function $S_n(\varphi,a)$ may be represented by a finite series involving the digamma functions only:
\be\notag
\begin{array}{ll}
\displaystyle 
\sum_{l=1}^{n-1}  \csc\!\left(\varphi+\frac{\,a\pi l\,}{n}\!\right) 
=\,\frac{\,1\,}{2\pi}\sum_{l=1}^{n-1}  &\displaystyle \left\{  
\Psi\!\left(\frac12+\frac{al}{2n}+\frac{\varphi}{2\pi}\right) - \Psi\!\left(\frac{al}{2n}+\frac{\varphi}{2\pi}\right) +  \right.\\[8mm]
&\displaystyle \quad\left.
+ \Psi\!\left(1+\frac{al}{2n}-\frac{\varphi}{2\pi}-\frac{a}{2}\right) - \Psi\!\left(\frac12+\frac{al}{2n}-\frac{\varphi}{2\pi}-\frac{a}{2}\right) 
\! \right\}\,.
\end{array}
\ee
where $\varphi$ and $a$ should satisfy conditions stated in \eqref{984ycbn492v2}.
For $S_n$ we have a much simpler expression:
\be\label{9782xdyb493yf}
\sum_{l=1}^{n-1}  \csc\frac{\,\pi l\,}{n}\,= \,\frac{1}{\,\pi\,}\!
\sum_{l=1}^{n-1} \left\{\Psi\!\left(\frac12+\frac{l}{2n}\right) - \Psi\!\left(\frac{l}{2n}\right) \!\right\}\,,
\ee
\end{theorem}

\begin{corollary}\label{9837dyi2hd1}
The function $S_n(\varphi,1)$ may also be represented by this expression of a mixed type:
\be\notag
\begin{array}{ll}
&\displaystyle\sum_{l=1}^{n-1}  \csc\!\left(\varphi+\frac{\,\pi l\,}{n}\!\right) 
 =\,-\frac{\,2n\ln n\,}{\pi} - \frac{\,2(n-1)\ln2\,}{\pi} +n\ctg{n\varphi} \, -\, \ctg\varphi \, + \\[8mm]
&\displaystyle \quad+
\frac{\,2\,}{\pi} \left\{ n\Psi\!\left(\frac{n\varphi}{\pi}\right) - \Psi\!\left(\frac{\varphi}{\pi}\right)\!\right\}
- \frac{\,1\,}{\pi}\!\sum_{l=1}^{n-1} \left\{\Psi\!\left(\frac{l}{2n}+\frac{\varphi}{2\pi}\right) 
+ \Psi\!\left(\frac{l}{2n}-\frac{\varphi}{2\pi}\right) \! \right\}\,,\qquad \frac{\varphi}{\pi}\notin\mathbbm{Q}\,.
\end{array}
\ee
For the case $\varphi=0$, i.e. for $S_n$, we have quite a different expression:
\be\notag
\sum_{l=1}^{n-1}  \csc\frac{\,\pi l\,}{n}\,= \, -\frac{\,2n\ln n\,}{\pi} - \frac{\,2(\gamma+\ln2)(n-1)\,}{\pi} 
-\,\frac{2}{\,\pi\,}\!\sum_{l=1}^{n-1} \Psi\!\left(\frac{l}{2n}\right) \,.
\ee
\end{corollary}

\begin{proof}
From the recurrence formula for the digamma function, it follows that 
the expression in curly braces in the right-hand side of the first formula of Theorem \ref{oiue2ynx} may be written as:
\be\notag
\Psi(n+\alpha_1)-\Psi(1+\alpha_1) +\Psi(\alpha_2)-\Psi(1+\alpha_2-n)\,
=\sum_{l=1}^{n-1}\left[\frac{1}{l+\alpha_1} + \frac{1}{l+\alpha_2-n}\right]\,,
\ee
where $\,\displaystyle\alpha_1\equiv\frac{nk}{a}+\frac{n\varphi}{a\pi}\,$ and 
$\,\displaystyle\alpha_2\equiv\frac{nk}{a}-\frac{n\varphi}{a\pi} +\frac{n}{a}\,.$
Hence
\be\notag
\begin{array}{ll}
\displaystyle 
\sum_{k=0}^{\infty} (-1)^k &\displaystyle \left\{  
\Psi\!\left(\frac{nk}{a}+n+\frac{n\varphi}{a\pi}\right) - \Psi\!\left(1+\frac{nk}{a}+\frac{n}{a}-\frac{n\varphi}{a\pi}-n\right) -  \right.\\[8mm]
&\displaystyle \quad\left.
- \Psi\!\left(1+\frac{nk}{a}+\frac{n\varphi}{a\pi}\right) + \Psi\!\left(\frac{nk}{a} +\frac{n}{a}-\frac{n\varphi}{a\pi}\right) \! \right\}
= \sum_{l=1}^{n-1} \sum_{k=0}^{\infty}  \left[\frac{(-1)^k}{l+\alpha_1} + \frac{(-1)^k}{l+\alpha_2-n}\right] \\[8mm]
&\displaystyle \quad
=\,\frac{a}{\,n\,}\! \sum_{l=1}^{n-1} \sum_{k=0}^{\infty} \left[\frac{(-1)^k}{\,k+\frac{al}{n}+\frac{\varphi}{\pi}\,} + 
\frac{(-1)^k}{\,k+\frac{al}{n}-\frac{\varphi}{\pi}+1-a\,} \right]\\[8mm]
&\displaystyle \quad
=\,\frac{a}{\,2n\,}\! \sum_{l=1}^{n-1} \left[ 
\Psi\!\left(\frac12+\frac{al}{2n}+\frac{\varphi}{2\pi}\right) - \Psi\!\left(\frac{al}{2n}+\frac{\varphi}{2\pi}\right) +\right.\\[8mm]
&\displaystyle \qquad\qquad\qquad\left.
+ \Psi\!\left(1+\frac{al}{2n}-\frac{\varphi}{2\pi}-\frac{a}{2}\right) - \Psi\!\left(\frac12+\frac{al}{2n}-\frac{\varphi}{2\pi}-\frac{a}{2}\right) 
\right]
\end{array}
\ee
since
\be\label{ch28gc2c}
\sum_{k=0}^{\infty} \frac{(-1)^k}{\,k+ b\,}\,=\, 
\frac{1}{\,2\,}\left\{\Psi\!\left(\frac12+\frac{b}{\,2\,}\right) - \Psi\!\left(\frac{b}{\,2\,}\right)\!\right\}\,,
\qquad b\in\mathbbm{C}\setminus\{0,-1,-2,\ldots\,\}. 
\ee
This completes the proof of the main formula of Theorem \eqref{oiue2ynx2}. 
The formula for $S_n$ is then obtained by setting $\varphi=0$ and $a=1.$ 

The formul\ae~in Corollary \ref{9837dyi2hd1} are obtained as follows. 
First note that $2\pi S_n(\varphi,1)=f_n(\varphi)+f_n(-\varphi),$ where
\be\notag
f_n(\varphi)\,\equiv
\sum_{l=1}^{n-1}  \! \left\{  \!
\Psi\!\left(\frac12+\frac{l}{2n}+\frac{\varphi}{2\pi}\right) - \Psi\!\left(\frac{l}{2n}+\frac{\varphi}{2\pi}\right) \!\right\}.
\ee
Now, applying twice Gauss' multiplication theorem for the digamma function we have
\be\notag
\begin{array}{ll}
\displaystyle 
f_n(\varphi)\; &\displaystyle =\, -2(n-1)\ln2 \,+ \,2 \!\sum_{l=1}^{n-1} \Psi\!\left(\frac{l}{n}+\frac{\varphi}{\pi}\right) 
- \,2 \!\sum_{l=1}^{n-1} \Psi\!\left(\frac{l}{2n}+\frac{\varphi}{2\pi}\right) \\[8mm]
&\displaystyle 
 =\, -2(n-1)\ln2 \,+ \,2n \Psi\!\left(\frac{n\varphi}{\pi}\right) - 2n\ln n - 2\Psi\!\left(\frac{\varphi}{\pi}\right) 
- \,2 \!\sum_{l=1}^{n-1} \Psi\!\left(\frac{l}{2n}+\frac{\varphi}{2\pi}\right)
\end{array}
\ee
Adding the latter expression to $f_n(-\varphi)$ and accounting for the fact that
\be\notag
\begin{array}{c}
\displaystyle \Psi\!\left(\frac{\varphi}{\pi}\right)  + \Psi\!\left(-\frac{\varphi}{\pi}\right)  = 
2\Psi\!\left(\frac{\varphi}{\pi}\right)  +\pi\ctg\varphi  +\frac{\pi}{\varphi}  \\[8mm]
\displaystyle n\Psi\!\left(\frac{n\varphi}{\pi}\right)  + n\Psi\!\left(-\frac{n\varphi}{\pi}\right)  = 
2n\Psi\!\left(\frac{n\varphi}{\pi}\right)  +\pi n \ctg n\varphi  +\frac{\pi}{\varphi} 
\end{array}
\ee
yields the first formula in Corollary \ref{9837dyi2hd1}. 
In order to obtain the second formula, we use again the duplication formula
for the digamma function and Gauss' multiplication theorem\footnote{See, e.g., \cite[Sec.~6.3, pp.~258--259]{abramowitz_01}, 
\cite[Sec.~1.7.1]{bateman_01}, \cite[Eq.~(77)]{iaroslav_07}.}:
\be\notag
\sum_{l=1}^{n-1} \left\{\Psi\!\left(\frac12+\frac{l}{2n}\right) + \Psi\!\left(\frac{l}{2n}\right)\! \right\}\,
=\,2\!\sum_{l=1}^{n-1} \left\{\Psi\!\left(\frac{l}{n}\right) -\ln2\right\} = \,-2n\ln n - 2(\gamma+\ln2)(n-1)\,.
\ee
Dividing the latter expression by $\pi$ and subtracting \eqref{9782xdyb493yf} from the result, 
we obtain the required formula for $S_n$. 
\end{proof}

\subsection{Asymptotic studies for large $n$}\label{4c312ct3q4rgfwa}
In this part we study the asymptotic behaviour of $S_n(\varphi,a)$ at $n\to\infty$. The results 
are based on the series representations obtained earlier
in Theorem \ref{oiue2ynx}. We already saw that the results may look quite differently for some particular cases.
In this section, we separately treat various cases, which, as we come to show, lead to qualitatively different 
behaviour of $S_n(\varphi,a)$ when $n$
starts to grow. Unlike previous cases, we begin with the most simple case, that of $S_n$.

\renewcommand\thetheorem{\arabic{theorem}a}
\subsubsection{Asymptotic expansions of $S_n$ for large $n$}
\begin{theorem}[Asymptotics of $\bm{S_n}$ via the $n$th harmonic number]\label{ordtj56jx}
Let $n_0$ and $N$ be natural numbers, such that $n_0$ is sufficiently large and $N\neq1.$ 
Then, for any $n>n_0$, the sum $S_n$
admits the following expansion, which at $n\to\infty$ becomes its complete (or full) asymptotics:
\be\notag
\begin{array}{ll}
\displaystyle 
\sum_{l=1}^{n-1}  \csc\frac{\,\pi l\,}{n}\, = & \displaystyle\,\frac{\,2n\,}{\pi}\left(H_{n} -\ln\frac{\,\pi\,}{2}\right)
-\frac{1}{\,\pi\,} +\\[6mm]
&\displaystyle \,
+\frac{1}{\,\pi\,}\!\sum_{r=1}^{N-1} \frac{B_{2r}}{\,r\,n^{2r-1}\,}
\left\{1-\,\frac{\,(-1)^{r+1}\,\big(2^{2r}-2\big)\,\pi^{2r} B_{2r}\,}{ (2r)!}  \right\}+ O\big(n^{1-2N}\big)\,,
\end{array}
\ee
where $N=2,3,4,\ldots$\,.
Writing down first few terms, we have
\be\notag
\begin{array}{ll}
\displaystyle 
\sum_{l=1}^{n-1}  \csc\frac{\,\pi l\,}{n}\,= & \displaystyle \,\,\frac{\,2n\,}{\pi}\left(H_{n} -\ln\frac{\,\pi\,}{2}\right)
-\frac{1}{\,\pi\,} -  \frac{1}{\!n\!}\!\left(\frac{\pi}{36} - \frac{1}{6\pi} \right) + 
 \frac{1}{\,n^3\,}\!\left(\frac{7\pi^3}{21\,600} -\frac{1}{60\pi}\right)  -\\[7mm]
 &\displaystyle \, - 
 \frac{1}{\,n^5\,}\!\left(\frac{31\pi^5}{1\,905\,120}  - \frac{1}{126\pi}   \right) + 
 \frac{1}{\,n^7\,}\!\left(\frac{127\pi^7}{72\,576\,000}-\frac{1}{120\pi} \right) -  \ldots
 \end{array}
\ee
\end{theorem}

\begin{proof}
We first recall an exact variant of the
Stirling formula for a real positive argument and a fixed number of terms:
\be\label{lk2093mffmnjw}
\Psi(x) \, =\,\ln x \,-\, \frac{1}{\,2x\,} \,-\, \frac{1}{\,2\,}\!\sum_{r=1}^{N-1} \frac{B_{2r}}{\,r\,x^{2r}\,}
\,-\, \theta\cdot\frac{B_{2N}}{\,2Nx^{2N}\,}\,, \qquad x>0\,, \quad 0<\theta<1\,, \quad N=2,3,4,\ldots\,.
\ee
For a sufficiently large $n_0$ and fixed $N=2,3,4,\ldots\,$, we, therefore, have:
\be\notag
\Psi(nk+n) -\Psi(nk) \,=\,\ln\left(1+\frac1k\right) + \,\frac{1}{\,2nk(k+1)\,}\,
-\, \frac{1}{\,2\,}\!\sum_{r=1}^{N-1} \frac{B_{2r}}{\,r\,n^{2r}\,} \left\{\frac{1}{\,(k+1)^{2r}\,}- \frac{1}{k^{2r}}\right\}
\,+\,O\Big(n^{-2N}\Big)\,,
\ee
where $n>n_0$.
Thus, by Theorem \ref{oiue2ynx} 
\be\notag
\begin{array}{ll}
\displaystyle 
S_n\;  & \displaystyle =
\frac{\,2nH_{n}\,}{\pi}
\,-\,\frac{\,2(1-\ln2)\,}{\pi}\,+\,\frac{\,2n\,}{\pi}\!\sum_{k=1}^{\infty} (-1)^k \ln\left(1+\frac1k\right) 
+\,\frac{1}{\pi}\sum_{k=1}^{\infty} \frac{(-1)^k}{\,k(k+1)\,} - \\[6mm]
&\displaystyle\qquad -\,\frac{1}{\,\pi\,}
\sum_{r=1}^{N-1} \frac{B_{2r}}{\,r\,n^{2r-1}\,} \sum_{k=1}^{\infty}\left\{\frac{ (-1)^k}{\,(k+1)^{2r}\,}- \frac{ (-1)^k}{k^{2r}}\right\}
+\,O\Big(n^{1-2N}\Big)= \\[6mm]
\displaystyle 
&\displaystyle =\frac{\,2nH_{n}\,}{\pi}
\,-\,\frac{\,2(1-\ln2)\,}{\pi}\,-\,\frac{\,2n\,}{\pi}\ln\frac\pi2
+\,\frac{\,1-2\ln2\,}{\pi} + \\[6mm]
&\displaystyle\qquad +\,\frac{1}{\,\pi\,}
\sum_{r=1}^{N-1} \frac{\big(1-2\eta(2r)\big) B_{2r}}{\,r\,n^{2r-1}\,}
+\,O\Big(n^{1-2N}\Big) = \\[6mm]
&\displaystyle = \frac{\,2n\,}{\pi}\left(H_{n} -\ln\frac{\,\pi\,}{2}\right) -\,\frac{1}{\,\pi\,} + \,\frac{1}{\,\pi\,}
\sum_{r=1}^{N-1} \frac{1-2\big(1-2^{1-2r}\big)\zeta(2r) }{\,r\,n^{2r-1}\,}\,B_{2r}
+\,O\Big(n^{1-2N}\Big)
\end{array}
\ee
for the same $N$ and $n$ as above. Using now the famous result due to Euler
\be\label{894yrnbcssw}
\zeta(2r)\,=\,\frac{(-1)^{r+1}(2\pi)^{2r} B_{2r}}{\,2\,(2r)!\,}\,,\qquad
r\in\mathbbm{N}\,,
\ee
we immediately arrive at the expansion announced in Theorem \eqref{ordtj56jx}. 
\end{proof}

\addtocounter{theorem}{-1}
\renewcommand\thetheorem{\arabic{theorem}b}
\begin{theorem}[Asymptotics of $\bm{S_n}$ via the logarithm: Watson 1916, Hargreaves 1922, Williams 1934,\ldots]
\label{ordtj56jx2}
Under the same conditions as in Theorem \ref{ordtj56jx}, the sum $S_n$ admits the following expansion
\be\notag
\begin{array}{ll}
\displaystyle 
\sum_{l=1}^{n-1}  \csc\frac{\,\pi l\,}{n}\,& \displaystyle  =  \,\frac{\,2n\,}{\pi}\!\left(\ln\frac{2n}{\pi}+\gamma \right)
+\frac{2}{\,\pi\,}\!\sum_{r=1}^{N-1} \frac{\,(-1)^{r} \big(2^{2r-1}-1\big)\pi^{2r} B^2_{2r}\,}{\,r\, (2r)!\, n^{2r-1}\,}
+ O\big(n^{1-2N}\big) \\[6mm]
\displaystyle 
& \displaystyle = \,\frac{\,2n\,}{\pi}\!\left(\ln\frac{2n}{\pi}+\gamma \right)
-\frac{\pi}{\,36\,n\,} + \frac{7\pi^3}{\,21\,600\,n^3\,} - \frac{\,31\pi^5\,}{\,1\,905\,120\,n^5\,}+\ldots
 \end{array}
\ee
which at $n\to\infty$ becomes its complete (or full) asymptotics.
This expression has the advantage not to contain the harmonic numbers, but 
is slightly less accurate than that obtained in Theorem \ref{ordtj56jx} (see Fig.~\ref{937fybfe4}).
\end{theorem}

\begin{proof}
The $n$th harmonic number may be written as
\be\label{093u2rxnc20u}
H_n\,=\,H_{n-1}+\frac1n\,=\,\Psi(n)+\gamma+\frac{1}{n}\,.
\ee
Using Stirling's formula \eqref{lk2093mffmnjw} for $\Psi(n)$ and making $n\to\infty$, we obtain
\be\label{jyhg6tv}
H_n\,= \ln n \,+\, \gamma \,+\, \frac{1}{\,2n\,} \,-\, \frac{1}{\,2\,}\!\sum_{r=1}^{N-1} \frac{B_{2r}}{\,r\,n^{2r}\,}
\,+ O\Big(n^{-2N}\Big)\,,\qquad n\to\infty\,,
\ee
for any $N=2,3,4,\ldots\,$ Substituting this formula into the expansion obtained in Theorem \eqref{ordtj56jx} yields the desired result.
\end{proof}

\renewcommand\thetheorem{\arabic{theorem}}

\noindent\textbf{Historical remark.}\label{8794rcb3894}
The expansion obtained in Theorem \ref{ordtj56jx2} is a known result.
We do not know when and by whom the asymptotics of $S_n$ was first derived,
but Williams \cite{williams_01} in 1934 points out that it may be readily obtained by the application 
of the Euler--Maclaurin summation formula.\footnote{Quite curiously, Chen in \cite[p.~80]{chen_06} states that the establishment of the asymptotics of
$S_n$ via the Euler--Maclaurin summation formula remains an open problem in 2010.} Hence, the first-order asymptotics of $S_n$ could be known even to the mathematicians 
of the XVIIIth century. As a matter of fact, we mention that the first-order asymptotics $S_n=O(n\ln n)$ was also found empirically 
by Sacha and Eckhardt in 2003 \cite{thiede_01}, \cite[p.~189]{sacha_01}.
As to the complete asymptotics of $S_n$, the earliest work we could find dates back to 1916 
and is due to the famous English mathematician George N.~Watson \cite{watson_02,watson_03}.\footnote{In the latter paper
Watson also considered more general series $\sum\limits_{l=1}^{n-1}\csc^r(\pi l/n)$ for positive integer $r$ and obtained their (first-order) asymptotics.
Some of these sums were also independently investigated by Grabner and Prodinger in 2007 \cite{grabner_01}.}
His result is the same as our Theorem \ref{ordtj56jx2}, except that Watson uses a slightly different definition for the Bernoulli numbers.
Besides Watson, the asymptotics of $S_n$ was also given by Hargreaves \cite{hargreaves_01} in 1922, by Williams \cite{williams_01} in 1934,
and probably by many others, because it is very difficult to verify whether this result is new or not. 
It is also interesting to compare how well Theorems \ref{ordtj56jx} and \ref{ordtj56jx2} approximate the value of $S_n$.
Figure \ref{937fybfe4} shows that both approximations are very good, and that the approximation given by Theorem \ref{ordtj56jx} is slightly better than that 
given by Theorem \ref{ordtj56jx2}.

\addtocounter{footnote}{-1}
\begin{figure}[!t]   
\centering
\includegraphics[width=0.48\textwidth]{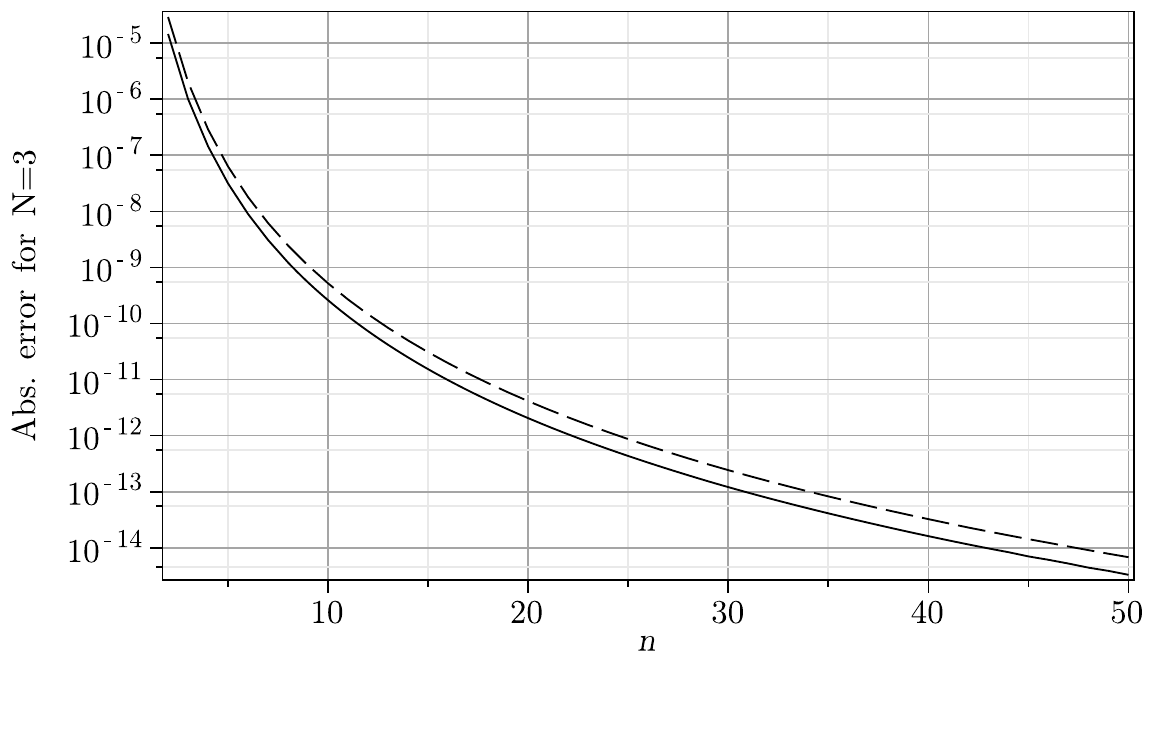}
\includegraphics[width=0.48\textwidth]{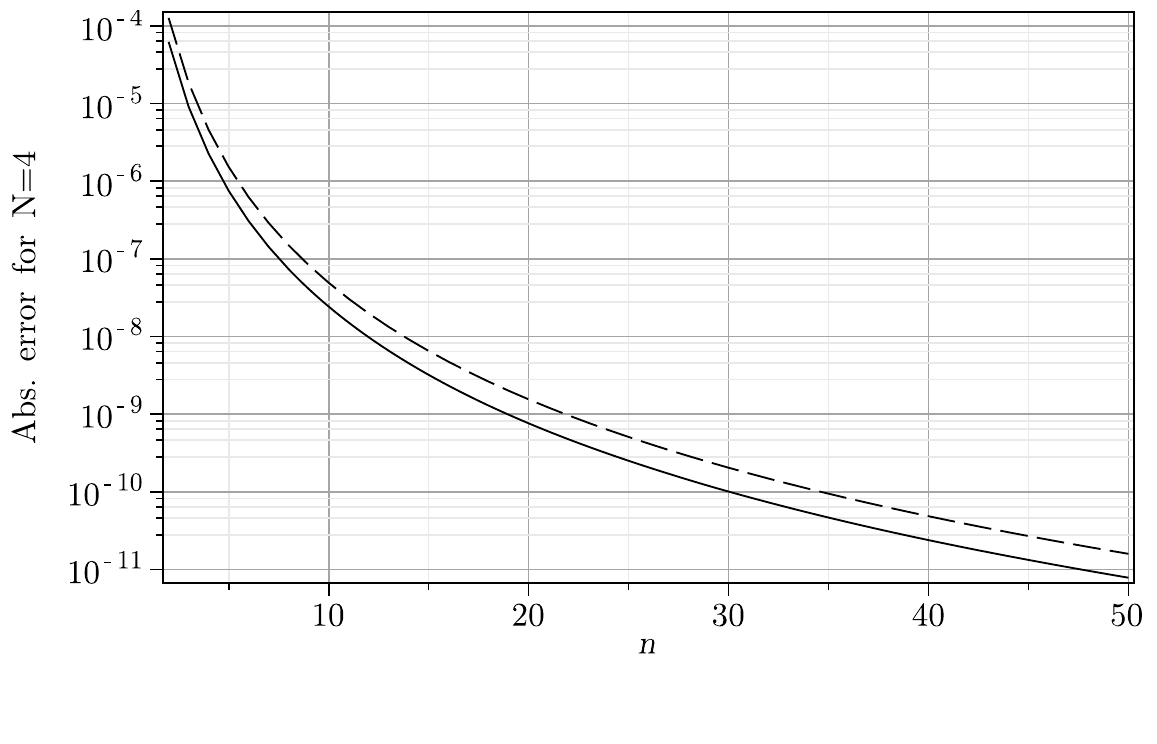}
\vspace{-1.5em}
\caption{The absolute error between $S_n$ and the approximations given by Theorems \ref{ordtj56jx} and \ref{ordtj56jx2}
for $N=3$ and  $N=4$.\protect\footnotemark{} The black solid line corresponds to the approximation given by Theorem \ref{ordtj56jx},
the black dashed line corresponds to that given by Watson's approximation (Theorem \ref{ordtj56jx2}).
Notice that the corresponding relative errors are even smaller, since $S_{10}\approx15.4$ and $S_{50}\approx129.$
}
\label{937fybfe4}
\end{figure}
\footnotetext{In both cases, terms corresponding to $O\big(n^{1-2N}\big)$ are taken equal to $0.$}

\subsubsection{Asymptotic expansion of $S_n(\varphi,a)$ for large $n$}
In order to investigate the behaviour of $S_n(\varphi,a)$ when $n\to\infty$, we first need to prove the following lemma.

\begin{lemma}\label{j873dhaq}
Let $\alpha$, $\beta$, $\beta_1$ and $\beta_2$ be arbitrary positive numbers.
Then, the following (asymptotical) expansion holds
\begin{eqnarray}
\displaystyle 
&&\sum_{k=1}^{\infty} (-1)^k    \displaystyle \Big\{ 
\Psi\big[n(\alpha k + \beta)\big]- \Psi\big(n\alpha k\big) \!\Big\}  = \,
\ln\left\{\frac{\,\alpha\,\Gamma\!\left(\frac12+\frac{\beta}{2\alpha}\right)}
{\beta\,\Gamma\!\left(\frac{\beta}{2\alpha}\right)} \right\}
+\,\ln2 - \frac12\ln\pi \, -    \notag\\[5mm]
&& \displaystyle\quad
- \,\frac{\ln2}{\,2\alpha n\,}
-\, \frac{1}{\,4\alpha n\,}\left\{\Psi\!\left(\frac12+\frac{\beta}{\,2\alpha\,}\right) - \Psi\!\left(\frac{\beta}{\,2\alpha\,}\right)
-\frac{2\alpha}{\beta}\right\} 	
- \frac{1}{\,2\,}\!\sum_{r=1}^{N-1} \frac{B_{2r}}{\,r\,(2\alpha n)^{2r} (2r-1)!\,} \times\notag \\[5mm]
&& \displaystyle
\times
\left\{\Psi_{2r-1}\!\left(\frac{\beta_1}{\,2\alpha\,}\right) 
- \,\Psi_{2r-1}\!\left(\frac12+\frac{\beta_1}{\,2\alpha\,}\right) -\,\frac{(2\alpha)^{2r} (2r-1)!}{\beta_1^{2r}}
+\,\frac{\,(2^{2r}-2)(2\pi)^{2r}(-1)^{r+1}B_{2r} \,}{4r}\right\}\notag\\[5mm]
&& \displaystyle\qquad
 + O\big(n^{-2N}\big) \,, \qquad n>n_0\,,\notag
\end{eqnarray}
where $N=2,3,4,\ldots$ and $n_0\in\mathbbm{N}$ is sufficiently large.

Furthermore, more generally we have
\begin{eqnarray}
\displaystyle 
\sum_{k=1}^{\infty} (-1)^k   && \displaystyle \Big\{ 
\Psi\big[n(\alpha k + \beta_1)\big]- \Psi\big[n(\alpha k + \beta_2)\big] \!\Big\}  = \,
\ln\left\{\frac{\,\Gamma\!\left(\frac12+\frac{\beta_1}{2\alpha}\right)\Gamma\!\left(\frac{\beta_2}{2\alpha}\right)\,}
{\Gamma\!\left(\frac12+\frac{\beta_2}{2\alpha}\right)\Gamma\!\left(\frac{\beta_1}{2\alpha}\right)} \right\}
-\,\ln\frac{\beta_1}{\beta_2} \,-  \notag \\[5mm]
&& \displaystyle
 \,-\, \frac{1}{\,4\alpha n\,}\left\{\Psi\!\left(\frac12+\frac{\beta_1}{\,2\alpha\,}\right) - \Psi\!\left(\frac{\beta_1}{\,2\alpha\,}\right)
- \Psi\!\left(\frac12+\frac{\beta_2}{\,2\alpha\,}\right) + \Psi\!\left(\frac{\beta_2}{\,2\alpha\,}\right) 
-\frac{2\alpha}{\beta_1} + \frac{2\alpha}{\beta_2}  \right\} 	- \notag	
\end{eqnarray}
\begin{eqnarray}
\displaystyle 
&& \displaystyle
- \frac{1}{\,2\,}\!\sum_{r=1}^{N-1} \frac{B_{2r}}{\,r\,(2\alpha n)^{2r} (2r-1)!\,} 
\left\{\Psi_{2r-1}\!\left(\frac{\beta_1}{\,2\alpha\,}\right) 
- \,\Psi_{2r-1}\!\left(\frac12+\frac{\beta_1}{\,2\alpha\,}\right) -\,\frac{(2\alpha)^{2r} (2r-1)!}{\beta_1^{2r}}  \, -\right. 	\notag\\[5mm]
&& \displaystyle\quad
 \left.-\Psi_{2r-1}\!\left(\frac{\beta_2}{\,2\alpha\,}\right) 
+\,\Psi_{2r-1}\!\left(\frac12+\frac{\beta_2}{\,2\alpha\,}\right) +\,\frac{(2\alpha)^{2r} (2r-1)!}{\beta_2^{2r}} \right\}
 + O\big(n^{-2N}\big) \,, \qquad n>n_0\,,\notag
\end{eqnarray}
provided that $\beta_1\neq\beta_2.$
\end{lemma}

\begin{proof}
Let $x_k>0$ and $y_k>0$ be two arbitrary sequences depending on $k$. 
By virtue of Stirling formula \eqref{lk2093mffmnjw}, for $n>n_0$, where $n_0\in\mathbbm{N}$ is sufficiently large, we have
\be\label{i2occu9}
\begin{array}{ll}
\displaystyle 
\sum_{k=1}^{\infty} (-1)^k  \Big\{ 
\Psi(n x_k)- \Psi(n y_k) \Big\}  = \: & \displaystyle\sum_{k=1}^{\infty} (-1)^k \ln\frac{x_k}{y_k} 
\,-\, \frac{1}{\,2n\,}\sum_{k=1}^{\infty} (-1)^k \left\{ \frac{1}{\,x_k\,}-\frac{1}{\,y_k\,}\right\} -\\[6mm]
 & \displaystyle
-\, \frac{1}{\,2\,}\!\sum_{r=1}^{N-1} \frac{B_{2r}}{\,r\,n^{2r}\,}
\sum_{k=1}^{\infty} (-1)^k \left\{ \frac{1}{\,x_k^{2r}\,}-\frac{1}{\,y_k^{2r}\,}\right\}
\,+O\big(n^{-2N}\big)\,=\,O(1)\,,
\end{array}
\ee
$N=2,3,4,\ldots\,$, provided that all series on the right converge. 
This implies, \emph{inter alia,} that independently of $n$ the infinite sum from the left--hand side remains always bounded.
Suppose now that $x_k$ and $y_k$ are linear in $k$ and have the same slope, i.e.
$x_k=\alpha k + \beta_1$ and $y_k=\alpha k + \beta_2$, $\beta_1\neq\beta_2$.
Using these formul\ae
\be
\begin{array}{l}
\displaystyle 
\sum_{k=1}^{\infty} (-1)^k\ln\frac{\alpha k+\beta_1}{\alpha k+\beta_2}\,= 
\sum_{k=1}^{\infty} (-1)^k\ln\!\left(1+\frac{\beta_1-\beta_2}{\alpha k+\beta_2}\right)= 
\,\ln\left\{\frac{\,\Gamma\!\left(\frac12+\frac{\beta_1}{2\alpha}\right)\Gamma\!\left(\frac{\beta_2}{2\alpha}\right)\,}
{\Gamma\!\left(\frac12+\frac{\beta_2}{2\alpha}\right)\Gamma\!\left(\frac{\beta_1}{2\alpha}\right)} \right\}
-\,\ln\frac{\beta_1}{\beta_2}\\[9mm]
\displaystyle 
\sum_{k=1}^{\infty} \frac{(-1)^k}{\alpha k+ \beta}\,=\, 
\frac{1}{\,2\alpha\,}\left\{\Psi\!\left(\frac12+\frac{\beta}{\,2\alpha\,}\right) - \Psi\!\left(\frac{\beta}{\,2\alpha\,}\right)\!\right\}
-\,\frac{1}{\beta}\,, \\[8mm]
\displaystyle 
\sum_{k=1}^{\infty} \frac{(-1)^k}{\big(\alpha k+ \beta\big)^{2r}}\,=\, 
\frac{1}{\,(2\alpha)^{2r} (2r-1)! \,}\left\{\Psi_{2r-1}\!\left(\frac{\beta}{\,2\alpha\,}\right) 
- \,\Psi_{2r-1}\!\left(\frac12+\frac{\beta}{\,2\alpha\,}\right) \!\right\}-\,\frac{1}{\beta^{2r}}\,,
 \end{array}
\ee
which can be obtained without any difficulty from, for example, \cite[p.~72, Exercise \no 23]{iaroslav_06}, 
formula \eqref{i2occu9} reduces to
\be\label{0293u3cxn30}
\begin{array}{lll}
&& \displaystyle 
\sum_{k=1}^{\infty} (-1)^k   \Big\{ 
\Psi\big[n(\alpha k + \beta_1)\big]- \Psi\big[n(\alpha k + \beta_2)\big] \!\Big\}  = \,
\ln\left\{\frac{\,\Gamma\!\left(\frac12+\frac{\beta_1}{2\alpha}\right)\Gamma\!\left(\frac{\beta_2}{2\alpha}\right)\,}
{\Gamma\!\left(\frac12+\frac{\beta_2}{2\alpha}\right)\Gamma\!\left(\frac{\beta_1}{2\alpha}\right)} \right\}
-\,\ln\frac{\beta_1}{\beta_2}   \,-\, \frac{1}{\,4\alpha n\,}\times \\[10mm]
&& \displaystyle \quad
\times\left\{\Psi\!\left(\frac12+\frac{\beta_1}{\,2\alpha\,}\right) - \Psi\!\left(\frac{\beta_1}{\,2\alpha\,}\right)
- \Psi\!\left(\frac12+\frac{\beta_2}{\,2\alpha\,}\right) + \Psi\!\left(\frac{\beta_2}{\,2\alpha\,}\right) 
-\frac{2\alpha}{\beta_1} + \frac{2\alpha}{\beta_2}  \right\} 		
- R_N\big(n;\alpha,\beta_1,\beta_2\big)\,,		
\end{array}
\ee
where $ N=2,3,4,\ldots$ and
\be\label{908rjnfc}
\begin{array}{ll}
&\displaystyle
R_N\big(n;\alpha,\beta_1,\beta_2\big)   \equiv\frac{1}{\,2\,}\!\sum_{r=1}^{N-1} \frac{B_{2r}}{\,r\,(2\alpha n)^{2r} (2r-1)!\,} 
\left\{\Psi_{2r-1}\!\left(\frac{\beta_1}{\,2\alpha\,}\right) 
- \,\Psi_{2r-1}\!\left(\frac12+\frac{\beta_1}{\,2\alpha\,}\right) -\Psi_{2r-1}\!\left(\frac{\beta_2}{\,2\alpha\,}\right) \right. +\\[8mm]
 & \displaystyle\qquad\qquad
 \left. 
+\,\Psi_{2r-1}\!\left(\frac12+\frac{\beta_2}{\,2\alpha\,}\right) 
-\,\frac{(2\alpha)^{2r} (2r-1)!}{\beta_1^{2r}}
+\,\frac{(2\alpha)^{2r} (2r-1)!}{\beta_2^{2r}}\!\right\}
 + O\big(n^{-2N}\big)
\,=\, O\big(n^{-2}\big)
\end{array}
\ee
for $n>n_0$. This is the second result of our Lemma.
Note that the remainder $ R_N\big(n;\alpha,\beta_1,\beta_2\big)$
behaves at large $n$ as $O\big(n^{-2}\big)$, and thus $nR_N\big(n;\alpha,\beta_1,\beta_2\big)$ vanishes as $n\to\infty.$

Now, if one of $\beta$'s equals zero, we may resort to an appropriated limiting procedure in order to evaluate 
the right--hand side of \eqref{0293u3cxn30}.
Bearing in mind that $\,\Gamma(1/2)=\sqrt\pi\,,$ $\,\Psi(1/2)=-\gamma-2\ln2\,$
and that $\,\Psi_{2r-1}(1/2)=\left(2^{2r}-1\right)(2r-1)!\,\zeta(2r)\,$,
and evaluating the following limits
\be\notag
\begin{array}{lll}
&& \displaystyle 
\lim_{z\to0}z\Gamma(z) =\lim_{z\to0}\Gamma(z+1) = 1\,, \qquad  \qquad
\lim_{z\to0}\left\{\!\Psi(z)+\frac1z \right\}=\lim_{z\to0}\Psi(z+1)=-\gamma \,,  \\[8mm]
&& \displaystyle 
\lim_{z\to0}\left\{\Psi_{2r-1}(z) - \frac{(2r-1)!}{z^{2r}}\right\}=
\lim_{z\to0}\Psi_{2r-1}(z+1) =\,(2r-1)!\cdot\zeta(2r)
\,=\,\frac{\,(2\pi)^{2r}(-1)^{r+1}B_{2r}\,}{4r}\,,
\end{array}
\ee
with $r\in\mathbbm{N}$, we immediately arrive at the first formula of the Lemma.
\end{proof}

\begin{theorem}[Asymptotics of $\bm{S_n(\varphi,a)}$]\label{iu2389ghi}
If, in addition to what was stated in \eqref{984ycbn492v2}, arguments $\varphi$ and $a$ are chosen so that
\be\label{79tvb89g7t}
0<\varphi<\pi \qquad \text{and}\qquad 1-\frac{\varphi}{\pi}<a<2-\frac{\varphi}{\pi}\,,
\ee
then there exists a sufficiently large $n_0$, such that for any $n>n_0$ the following equality holds:
\be\notag
\begin{array}{ll}
\displaystyle 
\sum_{l=1}^{n-1} \csc\!\left(\varphi+\frac{\,a\pi l\,}{n}\!\right) 
 =&\displaystyle \: \frac{n}{a}\ctg\frac{n(\pi-\varphi)}{a}
+ \frac{n}{a\pi}\ln\frac{\,\tg\left(\pi - \frac12\pi a- \frac12 \varphi\right)\,}{\tg\frac12\varphi } 
-\frac{\,\csc\varphi+\csc(\varphi+a\pi)\,}{2} -\\[8mm]
&\displaystyle 
- \sum_{r=1}^{N-1} \frac{\,(a\pi)^{2r-1}B_{2r}\,}{\,n^{2r-1} (2r)!\,} 
\Big\{\mathscr{F}_{2r-1}(\varphi) - \mathscr{F}_{2r-1}(\varphi + a\pi)  \Big\} 
+ O\big(n^{-2N+1}\big)\,,
\end{array}
\ee
where $N=2,3,4,\ldots\,,$ and
\be\notag
\mathscr{F}_{2r-1}(\alpha)\,\equiv\frac{d^{2r-1}\csc \alpha}{d\alpha^{2r-1}} \,.
\ee
At $n\to\infty$, the leading terms are placed from left to right in the first line;
the finite sum with the Bernoulli numbers tends to zero as $n\to\infty$. Note also that the sign of the tail, unlike
in the expansions from Theorems \ref{ordtj56jx}--\ref{ordtj56jx2}, varies and depends on the parameters 
$\varphi$ and $a$.\footnote{See Section \ref{2093u0j} where we study the 
sign of the general term of the sum with the Bernoulli numbers.}
\end{theorem}

\noindent\textbf{Remark.}
Using the reflection formul\ae~for the polygamma function, one may show that $\mathscr{F}_{2r-1}(\alpha)$
may also be computed via the polygamma functions:
\be\notag
\mathscr{F}_{2r-1}(\alpha) \,=\,
\frac{1}{\,(2\pi)^{2r}\,}\left\{ \Psi_{2r-1}\left(\frac12+\frac{\alpha}{2\pi}\right) 
+ \Psi_{2r-1}\left(\frac12-\frac{\alpha}{2\pi}\right)
-\Psi_{2r-1}\left(\frac{\alpha}{2\pi}\right) -  \Psi_{2r-1}\left(1-\frac{\alpha}{2\pi}\right)\!\right\}.
\ee
Furthermore, as we come to show on p.~\pageref{e3h498hc}, 
the difference $\,\mathscr{F}_{2r-1}(\varphi) - \mathscr{F}_{2r-1}(\varphi + a\pi)  \,$ 
may be given by the following integral
\be\notag
\mathscr{F}_{2r-1}(\varphi) - \mathscr{F}_{2r-1}(\varphi + a\pi)
\,=\,\frac{2^{2r+1}}{\pi^{2r}}\!\int\limits_{0}^{\infty} 
\!\frac{\, \ch\big[t(a-1)\big] \,}{\,\ch t\,} 
\sh\!\left[t\!\left(\frac{2\varphi}{\pi} + a-2\right)\right] t^{2r-1}\,  dt\,,
\ee
provided that conditions \eqref{79tvb89g7t} are satisfied. The latter may be useful for the study of sign
of $\,\mathscr{F}_{2r-1}(\varphi) - \mathscr{F}_{2r-1}(\varphi + a\pi)  \,$ for different $\varphi$ and $a$.

\begin{corollary}[Asymptotics of $\bm{S_n(\varphi,1)}$]\label{9uricxn34089}
If $\,0<\varphi<\pi\,$ and $\varphi\neq\pi k/n$, $k\in\mathbbm{Z}$, then
for $\,n>n_0\,$ the sum $S_n(\varphi,1)$ admits the following (asymptotic) expansion:
\be\notag
\sum_{l=1}^{n-1} \csc\!\left(\varphi+\frac{\,\pi l\,}{n}\!\right) 
 = \,-n\ctg\varphi n
- \frac{\,2n\,}{\pi}\ln\tg\frac\varphi2  
- \,2\!\sum_{r=1}^{N-1} \frac{\,\pi^{2r-1}B_{2r}\,}{\,n^{2r-1} (2r)!\,} 
\mathscr{F}_{2r-1}(\varphi) 
+ O\big(n^{-2N+1}\big)
\ee
with the same $\mathscr{F}_{2r-1}(\alpha)$, $N$ and $n_0$ as in Theorem \ref{iu2389ghi}.
At $n\to\infty$, first two terms in this expansion are leading, while the last sum with the Bernoulli numbers 
is $o(1)$.
\end{corollary}

\begin{proof}
In what follows, in addition to what was stated in \eqref{984ycbn492v2}, we also suppose that \eqref{79tvb89g7t} hold. 
First of all, since in the argument of the digamma functions in the right--hand side of \eqref{904cj9384nvc431u} we have a product of $n$ and $k$,
we need to sort out the zeroth term of the infinite sum. Furthermore, this zeroth term consists of four terms,
each of which should be transformed in such a way that all the arguments of $\Psi$ contain only a pure product of $n$ with some 
factor independent of $n.$ Using a recurrence formula for the digamma function, as well as a reflection formula, we have
\be\notag
\begin{array}{ll}
\displaystyle 
\Psi\!\left(n+\frac{n\varphi}{a\pi}\right) &\displaystyle \; - \: \Psi\!\left(1+\frac{n}{a}-\frac{n\varphi}{a\pi}-n\right)
- \Psi\!\left(1+\frac{n\varphi}{a\pi}\right) + \Psi\!\left(\frac{n}{a}-\frac{n\varphi}{a\pi}\right)  
= \Psi\!\left[n\!\left(1+\frac{\varphi}{a\pi}\right)\right] 
-\\[8mm]
&\displaystyle 
- \Psi\!\left[n\!\left(\frac{\varphi}{a\pi}+1-\frac1a\right)\right] 
- \Psi\!\left(\frac{n\varphi}{a\pi}\right)+ \Psi\!\left[n\!\left(\frac1a-\frac{\varphi}{a\pi}\right)\right] -  \frac{a\pi}{n\varphi}
-\pi \ctg\frac{n(\varphi-\pi)}{a} \,.
\end{array}
\ee
Now, we proceed similarly with the other four terms in the right-hand side of \eqref{904cj9384nvc431u}, which 
correspond to $k=1,2,3,\ldots$
We leave untouched the first and the fourth terms in curly braces in \eqref{904cj9384nvc431u}. 
The second term in braces is transformed as follows
\be\notag
\Psi\!\left(1+\frac{nk}{a}+\frac{n}{a}-\frac{n\varphi}{a\pi}-n\right) =
\Psi\!\left[n\!\left(\frac{k}{a}+\frac{1}{a}-\frac{\varphi}{a\pi}-1\right)\right]
+ \frac{a}{\,n\,}\cdot  \frac{1}{\,k+1-a-\frac{\varphi}{\pi}\,}\,.
\ee
Since we already accounted for all four terms in braces corresponding to $k=0,$ the overall contribution of the last term
in the above expression reduces to
\be\notag
\begin{array}{ll}
\displaystyle 
\frac{a}{\,n\,}\sum_{k=1}^\infty \; &\displaystyle\frac{(-1)^k}{\,k+1-a-\frac{\varphi}{\pi}\,}  =
\frac{a}{\,2n\,}\left\{\Psi\!\left(1-\frac{a}{2}-\frac{\varphi}{2\pi}\right)   -  
\Psi\!\left(\frac12-\frac{a}{2}-\frac{\varphi}{2\pi}\right) - \frac{2\pi}{\,\pi(1-a)-\varphi\,} \right\} \\[8mm]
&\displaystyle 
 \,=\,
\frac{a}{\,2n\,}\left\{\Psi\!\left(\frac{a}{2}+\frac{\varphi}{2\pi}\right)   -  
\Psi\!\left(\frac12-\frac{a}{2}-\frac{\varphi}{2\pi}\right) - \frac{2\pi}{\,\pi(1-a)-\varphi\,} 
+\pi\ctg\left(\frac{\varphi}{2}+\frac{\pi a}{2}\right)\right\} \\[8mm]
&\displaystyle 
 \,=\,
\frac{a}{\,2n\,}\left\{\Psi\!\left(\frac{a}{2}+\frac{\varphi}{2\pi}\right)   -  
\Psi\!\left(\frac12+\frac{a}{2}+\frac{\varphi}{2\pi}\right) - \frac{2\pi}{\,\pi(1-a)-\varphi\,} 
+2\pi\csc(\varphi+\pi a)\right\} \,,
\end{array}
\ee
where again we have used functional relationships for the digamma function, as well as formula \eqref{ch28gc2c}.
Proceeding analogously with the third term in curly braces, formula \eqref{904cj9384nvc431u} finally reduces to 
\be\label{0893uxhn2837}
\begin{array}{ll}
\displaystyle 
\sum_{l=1}^{n-1} &\displaystyle \:\csc\!\left(\varphi+\frac{\,a\pi l\,}{n}\!\right) 
 =\,-\csc(\varphi+a\pi) 
- \frac{n}{a}\ctg\frac{n(\varphi-\pi)}{a} + \,\frac{1}{\,\pi(1-a)-\varphi\,} \, +\\[8mm]
&\displaystyle
+ \frac{1}{\,2\pi\,}\left\{\Psi\!\left(\frac{\varphi}{2\pi}\right) 
- \Psi\!\left(\frac12+\frac{\varphi}{2\pi}\right)
+ \Psi\!\left(\frac12+\frac{a}{2}+\frac{\varphi}{2\pi}\right)
- \Psi\!\left(\frac{a}{2}+\frac{\varphi}{2\pi}\right)\! \right\}+
   \\[8mm]
&\displaystyle
+ \frac{\,n\,}{a\pi} \left\{ \Psi\!\left[n\!\left(\frac{\varphi}{a\pi}+1\right)\right] 
- \Psi\!\left[n\!\left(\frac{\varphi}{a\pi}+1-\frac1a\right)\right] 
+ \Psi\!\left[n\!\left(\frac1a-\frac{\varphi}{a\pi}\right)\right] 
- \Psi\!\left(\frac{n\varphi}{a\pi}\right)\!\right\} +
 \\[8mm]
&\displaystyle \qquad\qquad\qquad
+ \frac{\,n\,}{a\pi}\sum_{k=1}^{\infty} (-1)^k \left\{  
\Psi\!\left[n\!\left(\frac{k}{a}+1+\frac{\varphi}{a\pi}\right)\right] 
- \Psi\!\left[n\!\left(\frac{k}{a}+\frac{1}{a}-\frac{\varphi}{a\pi}-1\right)\right]  - \right.\\[8mm]
&\displaystyle \qquad\qquad\qquad\qquad\qquad\qquad\quad\left.
- \Psi\!\left[n\!\left(\frac{k}{a}+\frac{\varphi}{a\pi}\right)\right] 
+ \Psi\!\left[n\!\left(\frac{k}{a} +\frac{1}{a}-\frac{\varphi}{a\pi}\right)\right]\! \right\}\,.
\end{array}
\ee
The above expression, being certainly quite cumbersome, is nevertheless more suitable
for the study of the limiting case $n\to\infty$. In fact, since \eqref{79tvb89g7t} are fulfilled, 
all the arguments of the digamma functions appearing in the right--hand sides are srtictly positive,
and hence, we may now use Lemma \eqref{j873dhaq}. 
Applying the second result of Lemma \eqref{j873dhaq}
to the infinite sum over $k$ in \eqref{0893uxhn2837} yields after some algebra\footnote{We must use 
several times the reflection formul\ae~for the $\Gamma$--function and for its logarithmic 
derivative.}
\be\notag
\begin{array}{ll}
&\displaystyle
\sum_{k=1}^{\infty} (-1)^k \left\{  
\Psi\!\left[n\!\left(\frac{k}{a}+1+\frac{\varphi}{a\pi}\right)\right] 
- \Psi\!\left[n\!\left(\frac{k}{a}+\frac{1}{a}-\frac{\varphi}{a\pi}-1\right)\right] 
- \Psi\!\left[n\!\left(\frac{k}{a}+\frac{\varphi}{a\pi}\right)\right]  \right.+\\[8mm]
&\displaystyle \qquad\qquad\left.
+ \Psi\!\left[n\!\left(\frac{k}{a} +\frac{1}{a}-\frac{\varphi}{a\pi}\right)\right]\! \right\} =\,
\ln\!\left\{\frac{\varphi\big[\pi(1-a)-\varphi\big]}{(\varphi+a\pi)(\pi-\varphi)} \tg\left(\frac{\pi a}{2}+\frac\varphi2\right) \!\right\}
- \ln\tg\frac\varphi2  - \\[8mm]
&\displaystyle \qquad\qquad
-\frac{a}{2n}\left\{\Psi\!\left(\frac{\varphi}{\,2\pi\,}\right) -  \Psi\!\left(\frac12+\frac{\varphi}{\,2\pi\,}\right)
+  \Psi\!\left(\frac12+\frac{a}{2}+\frac{\varphi}{\,2\pi\,}\right)
-  \Psi\!\left(\frac{a}{2}+\frac{\varphi}{\,2\pi\,}\right)\!\right\}+\\[8mm]
\end{array}
\ee
\be\notag
\begin{array}{ll}
&\displaystyle \qquad\qquad
+ \pi \left[\csc\varphi -\csc(a\pi+\varphi)+\frac1\varphi - \frac{1}{\pi-\varphi} 
+ \frac{1}{\pi(1-a)-\varphi} - \frac{1}{a\pi+\varphi} \right]- \\[8mm]
&\displaystyle \qquad\qquad
- R_N\left(n;\frac1a,1+\frac{\varphi}{a\pi},\frac{\varphi}{a\pi}\right)
- R_N\left(n;\frac1a,\frac1a-\frac{\varphi}{a\pi},\frac1a-\frac{\varphi}{a\pi}-1\right)+ O\big(n^{-2N}\big)\,,
\end{array}
\ee
$N=2,3,4,\ldots$ Inserting the above expression into \eqref{0893uxhn2837} we obtain
\be\label{ouh3puhc}
\begin{array}{ll}
\displaystyle 
\sum_{l=1}^{n-1} \csc \,&\displaystyle \left(\!\varphi+\frac{\,a\pi l\,}{n}\!\right) 
 =\,-\frac12\Big\{\!\csc\varphi+\csc(\varphi+a\pi)\!\Big\}
 -\frac12\left\{\frac1\varphi - \frac{1}{\pi-\varphi} 
-  \frac{1}{a\pi+\varphi} - \frac{1}{\pi(1-a)-\varphi} \right\} +\\[8mm]
&\displaystyle 
+ \frac{\,n\,}{a\pi} \left\{ \Psi\!\left[n\!\left(\frac{\varphi}{a\pi}+1\right)\right] 
- \Psi\!\left[n\!\left(\frac{\varphi}{a\pi}+1-\frac1a\right)\right] 
+ \Psi\!\left[n\!\left(\frac1a-\frac{\varphi}{a\pi}\right)\right] 
- \Psi\!\left(\frac{n\varphi}{a\pi}\right) - \right.\\[8mm]
&\displaystyle
  -\pi \ctg\frac{n(\varphi-\pi)}{a}
+\ln\!\left\{\frac{\varphi\big[\pi(1-a)-\varphi\big]}{(\varphi+a\pi)(\pi-\varphi)} \tg\left(\frac{\pi a}{2}+\frac\varphi2\right) \!\right\}
- \ln\tg\frac\varphi2 - \\[6mm]
&\displaystyle \left. 
- R_N\left(n;\frac1a,1+\frac{\varphi}{a\pi},\frac{\varphi}{a\pi}\right) 
- R_N\left(n;\frac1a,\frac1a-\frac{\varphi}{a\pi},\frac1a-\frac{\varphi}{a\pi}-1\right)
 \right\}+ O\big(n^{-2N+1}\big)
\end{array}
\ee
Employing again the Stirling formula for the digamma functions from the second line and reassembling all the terms together, 
we see at once that
\be\label{8unfc09dff}
\begin{array}{ll}
\displaystyle 
\sum_{l=1}^{n-1} &\displaystyle \:\csc\!\left(\varphi+\frac{\,a\pi l\,}{n}\!\right) 
 =\, \frac{n}{a}\ctg\frac{n(\pi-\varphi)}{a}
+ \frac{n}{a\pi}\ln\frac{\,\tg\left(\pi - \frac12\pi a- \frac12 \varphi\right)\,}{\tg\frac12\varphi } 
-\frac{\,\csc\varphi+\csc(\varphi+a\pi)\,}{2} + \\[8mm]
&\displaystyle 
 + \frac{1}{\,2\,}\!\sum_{r=1}^{N-1} \frac{(a\pi)^{2r-1}B_{2r}}{\,r\,n^{2r-1}\,} 
\left\{\frac{1}{\varphi^{2r}} -  \frac{1}{(a\pi+\varphi)^{2r}} 
+ \frac{1}{(\pi(a-1)+\varphi)^{2r}} - \frac{1}{(\pi-\varphi)^{2r}} \right\}  -\\[8mm]
&\displaystyle
- \frac{\,n\,}{a\pi} \left\{
R_N\left(n;\frac1a,1+\frac{\varphi}{a\pi},\frac{\varphi}{a\pi}\right) 
+ R_N\left(n;\frac1a,\frac1a-\frac{\varphi}{a\pi},\frac1a-\frac{\varphi}{a\pi}-1\right)\!
 \right\}+ O\big(n^{-2N+1}\big)\,.
\end{array}
\ee
The latter expression, with the aid of \eqref{908rjnfc} and with the help of the combined reflection formula for 
the polygamma functions 
\be\notag
\Psi_{2r-1}\Big(\tfrac12+z\Big) + \Psi_{2r-1}\Big(\tfrac12-z\Big) 
-\Psi_{2r-1}(z) -  \Psi_{2r-1}(1-z)\,=\,2\pi\,\frac{d^{2r-1}\csc2\pi z}{dz^{2r-1}} \,,
\ee
simplifies and reduces to the expression stated in the Theorem. The latter, in turn, may be further 
simplified for some particular cases. For instance, setting $a=1$ we see that several terms 
vanish or/and coincide with some other terms. Corollary \ref{9uricxn34089} is the result
of such a simplification.
\end{proof}

\noindent\textbf{Historical remark.}\label{jf308j3}
It is quite remarkable that all the polygamma functions finally disappeared 
from the asymptotical expansion of $S_n(\varphi,a)$. Interestingly, Watson \cite{watson_02}
also tried to obtain a suitable formula for the asymptotics of the particular case $S_n(\varphi,1)$,
but his result differs a lot from our Corollary \ref{9uricxn34089}.
Watson's formula for $T_n$, which is equal to our $S_n(\varphi,1)$
at $\varphi=-\tfrac12\beta$, see \cite[pp.~116--118]{watson_02}, 
contains several digamma and polygamma functions, and does not clearly indicate the leading term. 
In fact, looking at the asymptotic formula for $\pi T_n$ \cite[p.~118, second formula]{watson_02}, it may seem that the leading 
term is $2n\ln n$, which is not true.
It is actually possible to further simplify Watson's expression in order to 
ascertain that not only the asymptotical expansion for $T_n$ may be written without polygamma functions,
but also that the first term in his asymptotical expansion $2n\ln n$ disappears.

\renewcommand\thetheorem{\arabic{theorem}a}
\begin{theorem}[Asymptotics of $\bm{S_n(0,a)}$ via the $n$th harmonic number]\label{iu2389ghiv333}
If $0<a<2$, $a\neq1$, $a\neq n/k$, $k\in\mathbbm{Z}\,$, then we have for $S_n(0,a)$:
\be\notag
\begin{array}{ll}
\displaystyle 
&\displaystyle  \sum_{l=1}^{n-1} \csc\frac{\,a\pi l\,}{n}
 = \frac{n\delta_a}{a}\ctg\frac{\pi n}{a} \,
+\,\frac{\,nH_{n}\,}{a\pi}\, 
+ \frac{n}{a\pi}\ln\left|\tg\tfrac12\pi a\right|
+ \frac{n}{a\pi}\ln\frac{2}{\pi a}
-\frac{1}{\,2\pi a\,} -\frac{\,\csc a\pi\,}{2} +\\[8mm]
&\displaystyle \;\;
+\frac{n}{\,a\pi\,} \left[\sum_{r=1}^{N-1} \frac{\,(a\pi)^{2r}B_{2r}\,}{\,n^{2r} (2r)!\,} 
\left\{\mathscr{F}_{2r-1}(a\pi) -\,\frac{\,(2^{2r}-2)(-1)^{r+1}B_{2r} \,}{2r} \right\} 
+\frac12\sum_{r=1}^{N-1} \frac{\,B_{2r}\,}{\,r\,n^{2r} \,}\right]
+ O\big(n^{-2N+1}\big)\,,
\end{array}
\ee
for $n>n_0$, where $n_0$ is a sufficiently large positive integer, $\mathscr{F}_{2r-1}(\alpha)$ is 
the same as in Theorem \ref{iu2389ghi} and
\be\notag
\delta_a = 
\begin{cases}
0\,, \qquad & 0<a<1  \,,\\[1mm]
1\,, \qquad & 1<a<2\,.\footnotemark
\end{cases}
\ee
\footnotetext{We, of course, could write $\lfloor a \rfloor $ 
instead of $\delta_a$, but in order to avoid any confusion on the direct applicability of the formula 
to other domains of $a$, we prefer to use this notation.}
\end{theorem}

\begin{proof}
The proof is carried out in the manner analogous to the preceding Theorem.
Starting from \eqref{904cj9384nvc431u} and accounting for \eqref{093u2rxnc20u}, 
we first reduce the former expression to the following form
\be\label{5rtgiy}
\begin{array}{ll}
\displaystyle 
\sum_{l=1}^{n-1} &\displaystyle \:\csc\frac{\,a\pi l\,}{n}\,
=\,-\csc a\pi + \frac{n}{a}\ctg\frac{\pi n}{a} \,
+\,\frac{\,nH_{n}\,}{a\pi}\, + \,\frac{2a-1}{\,\pi a(1-a)\,} +\,\frac{\,\ln2\,}{\pi}\,+\\[8mm]
&\displaystyle \qquad\qquad
+ \frac{1}{\,2\pi\,}\left\{\Psi\!\left(\frac{1+a}{2}\right)
- \Psi\!\left(\frac{\,a\,}{2}\right) \right\}
+ \frac{\,n\,}{a\pi} \left\{ \Psi\!\left(\frac{n}{a}\right)
- \Psi\!\left(\frac{n(a-1)}{a}\right)\!\right\} +   \\[8mm]
&\displaystyle \qquad\qquad
+ \frac{\,n\,}{a\pi}\sum_{k=1}^{\infty} (-1)^k \left\{  
\Psi\!\left(\frac{n(k+a)}{a}\right) - \Psi\!\left(\frac{n(k+1-a)}{a}\right)
- \Psi\!\left(\frac{nk}{a}\right) + \Psi\!\left(\frac{n(k+1)}{a}\right)\! \right\}\,.
\end{array}
\ee
$\,a\neq n/k$, $\,k\in\mathbbm{Z}\,$.\footnote{We suppress some details in order to shorten the proof.}
If $a>1$, we may directly use \eqref{lk2093mffmnjw} for the last expression in curly brackets
from the second line:
\be\notag
\Psi\!\left(\frac{n}{a}\right)
- \Psi\!\left(\frac{n(a-1)}{a}\right) =\,-\ln(a-1) \,-\, \frac{a^2-2a}{\,2n(a-1)\,} \,-\, \frac{1}{\,2\,}
\!\sum_{r=1}^{N-1} \frac{\,a^{2r}B_{2r}}{\,r\,n^{2r}\,} \Big\{1-(a-1)^{-2r} \Big\} +O\big(n^{-2N}\big)\,,
\ee
where $N=2,3,4,\ldots\,,$, $n>n_0$ and $n_0$ is a sufficiently large positive integer.
Substituting this expression into \eqref{5rtgiy} and employing both results of Lemma \ref{j873dhaq}
(note that $a$ should be lesser than 2), we arrive at the desired result for the case $1<a<2$. 

If, in contrast, $0<a<1$, we first transform the above--mentioned expression in curly brackets with the aid of the reflection formula
\be\notag
\Psi\!\left(\frac{n}{a}\right) - \Psi\!\left(\frac{n(a-1)}{a}\right) = 
\Psi\!\left(\frac{n}{a}\right) - \Psi\!\left(\frac{n(1-a)}{a}\right)
- \pi\ctg\frac{\pi n}{a}  - \frac{a}{\,n (1-a)\,}  
\ee
Employing then Stirling formula \eqref{lk2093mffmnjw}, we have
\be\notag
\Psi\!\left(\frac{n}{a}\right)
- \Psi\!\left(\frac{n(1-a)}{a}\right) =\,-\ln(1-a) \,-\, \frac{a^2}{\,2n(a-1)\,} \,-\, \frac{1}{\,2\,}
\!\sum_{r=1}^{N-1} \frac{\,a^{2r}B_{2r}}{\,r\,n^{2r}\,} \Big\{1-(a-1)^{-2r} \Big\} +O\big(n^{-2N}\big)\,,
\ee
Hence, accounting for the factor $n/(a\pi)$, the difference of contributions between the cases $1<a<2$ and $0<a<1$ is simply 
$-(n/a)\ctg(\pi n/a)-\ln(-1)$. 
Inserting this difference into the final formula for the case $1<a<2$,
we obtain the result of the theorem for the case $0<a<1$. 
\end{proof}

\addtocounter{theorem}{-1}
\renewcommand\thetheorem{\arabic{theorem}b}
\begin{theorem}[Asymptotics of $\bm{S_n(0,a)}$ via the logarithm]
Under the same conditions as in Theorem \ref{iu2389ghiv333} and with the same notation, the sum $S_n(0,a)$ 
admits the following (asymptotic) expansion
\be\notag
\begin{array}{ll}
\displaystyle 
\sum_{l=1}^{n-1} \csc\frac{\,a\pi l\,}{n}
 =&\displaystyle \;  \frac{n\delta_a}{a}\ctg\frac{\pi n}{a} \,
+ \frac{n}{a\pi}\left(\ln\frac{2n}{\pi a}+\gamma+\ln\left|\tg\tfrac12\pi a\right|\!\right)
-\frac{\,\csc a\pi\,}{2} +\\[8mm]
&\displaystyle 
+ \sum_{r=1}^{N-1} \frac{\,(a\pi)^{2r-1}B_{2r}\,}{\,n^{2r-1} (2r)!\,} 
\left\{\mathscr{F}_{2r-1}(a\pi)  -\,\frac{\,(2^{2r}-2)(-1)^{r+1}B_{2r} \,}{2r} \right\} 
+ O\big(n^{-2N+1}\big)\,,
\end{array}
\ee
\end{theorem}

\begin{proof}
Substituting expansion \eqref{jyhg6tv} into Theorem \ref{iu2389ghiv333} immediately yields
the required result.
\end{proof}

\renewcommand\thetheorem{\arabic{theorem}}

\subsubsection{Unification of the results on the asymptotic expansion of $S_n(\varphi,a)$ for large $n$}\label{08r7ch3489y}
The main results obtained in the preceding sections may be summarised as follows 
\be\notag
\begin{array}{ll}
&\displaystyle 
\sum_{l=1}^{n-1} \csc\!\left(\varphi+\frac{\,a\pi l\,}{n}\!\right) =   \\[10mm]
&\displaystyle \qquad
=\begin{cases}
\displaystyle \frac{\,2n\,}{\pi}\!\left(\ln\frac{2n}{\pi}+\gamma \right)
+ o(1)\,, & \varphi=0\,, \quad a=1\,, \\[7mm]
\displaystyle \,-n\ctg\varphi n  - \frac{\,2n\,}{\pi}\ln\tg\frac\varphi2  
+ o(1)\,, & 0<\varphi<\pi\,, \quad a=1\,, \\[7mm]
\displaystyle 
\frac{n}{a\pi}\left(\ln\frac{2n}{\pi a}+\gamma+\ln\tg \tfrac12\pi a\!\right)
-\frac{\,\csc a\pi\,}{2} +o(1)\,,  & \varphi=0\,, \quad 0<a<1\\[7mm]
\displaystyle 
\frac{n}{a}\ctg\frac{\pi n}{a} \,
+ \frac{n}{a\pi}\left[\ln\frac{2n}{\pi a}+\gamma+\ln\tg\left(\pi - \tfrac12\pi a\right)\!\right] - &   \\[7mm]
\displaystyle \displaystyle \qquad\qquad\qquad
-\frac{\,\csc a\pi\,}{2} +o(1)\,,  & \varphi=0\,, \quad 1<a<2\\[7mm]
\displaystyle   \frac{n}{a}\ctg\frac{n(\pi-\varphi)}{a}
+ \frac{n}{a\pi}\ln\frac{\,\tg\left(\pi - \frac12\pi a- \frac12 \varphi\right)\,}{\tg\frac12\varphi } - & \\[7mm]
\displaystyle \qquad\qquad\qquad
-\frac{\,\csc\varphi+\csc(\varphi+a\pi)\,}{2} +o(1)\,, & 0<\varphi<\pi\,, \quad 
1-\dfrac{\varphi}{\pi}<a<2-\dfrac{\varphi}{\pi}
\end{cases}
\end{array}
\ee
for $n\to\infty$. 
We see that the leading terms in the asymptotics of $S_n(\varphi,a)$ 
vary depending on the regions of $\varphi$ and $a$ where the sum $S_n(\varphi,a)$ is evaluated
(Table \ref{tr5b56wegd} summarizes the presence of different leading terms as a function of arguments $\varphi$ and $a$).
This explains why for some values of $\varphi$ and $a$ we only have a monotonic growth with $n$, 
while for other values of $\varphi$ and $a$ the sum $S_n(\varphi,a)$ becomes sporadically large, 
the phenomenon that we observed empirically on p.~\pageref{g6hytbhfw}, see Fig.~\ref{g6hytbhfw}.
This depends, of course, on the cotangent term, which is the only responsible for such large values,
other terms like $O(n\ln n)$ or $O(n)$ being almost negligible.\footnote{Furthermore, the sign of these negligible terms 
(the sign of each term of the tail) also changes from one region of $\varphi$ and $a$ to another.}
In simple terms, as long as the argument of the cosecant stays far away from its poles, 
the sum grows quite slowly. If the argument cannot reach even its first pole at $\pi$, e.g.~when $\varphi=0$ and $a\in(0,1)$,
then the growth of $S_n(\varphi,a)$ is simply monotonic as $O(n\ln n)$, 
see, e.g., the behaviour of $S_n(0,1)$ and of $S_n(0,\ln2)$ in Fig.~\ref{g6hytbhfw}.
In contrast, the more the argument approaches the poles of the cosecant at $\pi n$, $n\in\mathbbm{Z}$,
the more the contribution of the cotangent term is important. 
\footnote{It may also be suspected 
that the behaviour of this sum has something to do with the irrationality measure of $\pi$, $\mu(\pi)$, 
which is currently (as of mid-2023) not known 
and is estimated $\,2\leqslant\mu(\pi)<7.103205334138\,$.} 

The asymptotic expansion of at large $n$is not only interesting in itself, but also because
its particular cases, except the second one, cannot be deduced from the more general 
cases. For instance, it can be readily verified that 
\be\notag
\lim_{\varphi\to0^+}\!\left\{ -n\ctg\varphi n  - \frac{\,2n\,}{\pi}\ln\tg\frac\varphi2   \right\} \neq
\frac{\,2n\,}{\pi}\!\left(\ln\frac{2n}{\pi}+\gamma \right)
\ee
or that
\be\notag
\lim_{a\to1^-}\!\left\{ \frac{n}{a\pi}\left(\ln\frac{2n}{\pi a}+\gamma+\ln\tg \tfrac12\pi a\!\right)
-\frac{\,\csc a\pi\,}{2}  \right\} \neq
\frac{\,2n\,}{\pi}\!\left(\ln\frac{2n}{\pi}+\gamma \right)
\ee
both at large $n.$ This is the reason for which Euler's results, that we mentioned in the historical remark on p.~\pageref{rem5d45sx},
were so different from Watson's formula given in Theorem \ref{ordtj56jx2}:
first, they considered $S_n(\varphi,a)$ in different regions, and second, particular cases cannot be obtained 
from the more general cases.

Finally, we note that the above results may be extended to other domains of $\varphi$ and $a$ via the functional 
relationships given in Section \ref{h09387rxhxdws}.

\begin{table}[t!]
\centering
\begin{tabular}{|l|c|c|c|c|}
\hline
\textbf{Choice of} $\,\boldsymbol\varphi\,$ \textbf{and} $\,\boldsymbol a\,$ 	& $\boldsymbol{O(n\ctg (\alpha n - \beta))}$	& $\boldsymbol{O(n\ln n)}$ & $\boldsymbol{O(n)}$ & $\boldsymbol{O(1)}$	 \\[1mm]
\hline
\hline
$\varphi=0\,, \quad a=1	$	&			0		&   1					& 1 &	0\\[1mm]              
$0<\varphi<\pi\,, \quad a=1$	&			1		&   0					& 1 &	0\\[1mm]              
$\varphi=0\,, \quad 0<a<1	$	&			0		&   1					& 1 &	1\\[1mm]              
$\varphi=0\,, \quad 1<a<2 $	&			1		&   1					& 1 &	1\\[1mm]              
$0<\varphi<\pi\,, \quad 
1-\frac{\varphi}{\pi}<a<2-\frac{\varphi}{\pi}$	&			1		&   0					& 1 &	1\\[1mm]              
\hline
\end{tabular}
\caption{The presence of various leading terms in the asymptotical expansion of $S_n(\varphi,a)$ at large $n$
as a function of regions of $\varphi$ and $a$ in which $S_n(\varphi,a)$ is evaluated ($1$ stands for the presence,
$0$ for the absence, $\alpha$ and $\beta$ denote some coefficients independent of $n$)
\hfill\hfill}
\label{tr5b56wegd}
\end{table}

\subsection{Bounds for $S_n(\varphi,a)$}\label{2093u0j}

\begin{theorem}[Bounds for $\bm{S_n}$]\label{oiehrhg4w5h}
For the sum $S_n$, $n=2,3,4,\ldots\,$, the following upper and lower bounds hold:
\be\label{o4uincf87nb4}
-  \frac{A}{\! n\!} < \sum_{l=1}^{n-1}  \csc\frac{\,\pi l\,}{n}\, 
- \,\frac{\,2n\,}{\pi}\left(H_{n} -\ln\frac{\,\pi\,}{2}\right) + \, \frac{1}{\,\pi\,} < 
 -  \frac{A}{\! n\!}  +\,\frac{B}{\,n^3\,}
\ee
where 
\be\notag
A\equiv \frac{\pi}{36} - \frac{1}{6\pi} = 0.0342\ldots  \qquad{and}  \qquad
B\equiv \frac{7\pi^3}{21\,600} -\frac{1}{60\pi} = 0.00474\ldots
\ee
If it is not possible to make use of the harmonic numbers, then simpler bounds may also be used
\be\label{8438394yrxn}
- \frac{C}{\,n\,}< \sum_{l=1}^{n-1}  \csc\frac{\,\pi l\,}{n} - \frac{\,2n\,}{\pi}\!\left(\ln\frac{2n}{\pi}+\gamma \right)<
-\frac{C}{\,n\,} +\frac{D}{\,n^3\,}  
\ee
where
\be\notag
C\equiv \frac{\pi}{36} = 0.0872\ldots  \qquad{and}  \qquad
D\equiv \frac{7\pi^3}{21\,600} = 0.0100\ldots
\ee
The bounds with the harmonic numbers are slightly more accurate than those with the logarithm
(see Fig.~\ref{89723yxb29} for the detailed comparisons et some numeric examples).
\end{theorem}

\begin{proof}
For the proof, we make use of the asymptotic expansion for $S_n$ obtained in Theorem \ref{ordtj56jx}.
Consider the general term of the series on the right--hand side. Accounting for the fact that $\big|B_{2r}\big|=(-1)^{r+1} B_{2r}$,
we may rewrite it in a slightly different form
\be\notag
\frac{B_{2r}}{\,r\,n^{2r-1}\,}
\left\{1-\,\frac{\,(-1)^{r+1}\,\big(2^{2r}-2\big)\,\pi^{2r} B_{2r}\,}{ (2r)!}  \right\} = 
\frac{(-1)^r\big|B_{2r}\big|}{\,r\,n^{2r-1}\,}
\left\{\frac{\,\big(2^{2r}-2\big)\,\pi^{2r} \big|B_{2r}\big|\,}{ (2r)!} - \,1 \right\}  .
\ee
Such a writing suggests that the last expression in curly brackets remains, probably, always positive. 
Indeed, for $r=1$, we have:
\be\notag
\left.\frac{\,\big(2^{2r}-2\big)\,\pi^{2r} \big|B_{2r}\big|\,}{ (2r)!} - \,1  \,\right|_{r=1} \!\!\! 
=\underbrace{\frac{\pi^2}{6}-1}_{0.644\ldots} >0 \,.
\ee
For $r\geqslant2$, we may bound from below this terms as follows. By virtue of \eqref{894yrnbcssw}, we have
\be\notag
\frac{\,\big(2^{2r}-2\big)\,\pi^{2r} \big|B_{2r}\big|\,}{ (2r)!} - \,1  = \, 2\zeta(2r)-1- \frac{\zeta(2r)}{2^{2r-2}} > 
1 -\frac{\zeta(2r)}{2^{2r-2}} 
\geqslant \underbrace{1 -\frac{\pi^4}{\,360\,} }_{0.729\ldots} \,,\qquad r\geqslant2\,,
\ee
since $\,1<\zeta(s)\leqslant\frac{1}{90}\pi^4\,$ for any $s\geqslant4$. Hence, 
\be
0.644<\frac{\,\big(2^{2r}-2\big)\,\pi^{2r} \big|B_{2r}\big|\,}{ (2r)!} - \,1\,, \qquad r\in\mathbbm{N}\,,
\ee
and thus, the above--mentioned expression in curly brackets is always positive.
But this implies that the asymptotic expansion for $S_n$ obtained in Theorem \ref{ordtj56jx}
is alternating and fulfils the conditions related to the so--called \emph{enveloping series}, see e.g.~\cite[Chapt.~4, \S1]{polya_01_eng}.
Such series have the property to converge (if they do) to the value, which lies
between two consecutive partial sums. Therefore
\be\notag
\begin{array}{ll}
\displaystyle 
\frac{1}{\,\pi\,}\!\sum_{r=1}^{N-1} \frac{(-1)^r\big|B_{2r}\big|}{\,r\,n^{2r-1}\,}& \displaystyle\,
\left\{\frac{\,\big(2^{2r}-2\big)\,\pi^{2r} \big|B_{2r}\big|\,}{ (2r)!} - \,1 \right\}   \lessgtr
\sum_{l=1}^{n-1}  \csc\frac{\,\pi l\,}{n}\, - \frac{\,2n\,}{\pi}\left(H_{n} -\ln\frac{\,\pi\,}{2}\right)
+\frac{1}{\,\pi\,}\\[8mm]
&\displaystyle 
\lessgtr\frac{1}{\,\pi\,}\!\sum_{r=1}^{N} \frac{(-1)^r\big|B_{2r}\big|}{\,r\,n^{2r-1}\,}
\left\{\frac{\,\big(2^{2r}-2\big)\,\pi^{2r} \big|B_{2r}\big|\,}{ (2r)!} - \, 1 \right\} \, .
\end{array}
\ee
Setting in the above inequality $N=2$ and accounting for the sign yields
\be
\begin{array}{ll}
\displaystyle 
-  \frac{1}{\!n\!}\!\left(\frac{\pi}{36} - \frac{1}{6\pi} \right) & \displaystyle\, <
\sum_{l=1}^{n-1}  \csc\frac{\,\pi l\,}{n}\, - \frac{\,2n\,}{\pi}\left(H_{n} -\ln\frac{\,\pi\,}{2}\right)
+\frac{1}{\,\pi\,}\\[8mm]
&\displaystyle \qquad\qquad
<  -  \frac{1}{\!n\!}\!\left(\frac{\pi}{36} - \frac{1}{6\pi} \right) + 
 \frac{1}{\,n^3\,}\!\left(\frac{7\pi^3}{21\,600} -\frac{1}{60\pi}\right) ,
\end{array}
\ee
which is identical to \eqref{o4uincf87nb4}.
By a similar line of reasoning applied to Theorem \eqref{ordtj56jx2}, we obtain \eqref{8438394yrxn}.
\end{proof}

\noindent\textbf{Historical remark.}\label{u6rfv65}
Various bounds for $S_n$ are available in the mathematical literature. 
Below are just a few recent examples.
Cochrane in 1987 \cite[Lemma 3.1]{cochrane_01}, and also Peral in 1990 \cite[Lemma 4]{peral_01}, 
proved that 
\be\notag
\sum_{l=1}^{n-1}  \csc\frac{\,\pi l\,}{n}\,= \,\frac{\,2n\,}{\pi}\!\left(\ln\frac{2n}{\pi}+\gamma \right)+ O(1)\,,
\qquad n=2,3,4,\ldots
\ee 
Kongting \cite[Lemma 1]{kongting_01}\footnote{The reader, however, should be careful with this paper, 
because it also contains several incorrect results. For instance, Theorem 1 and Lemma 2 are wrong 
(incorrect results in this work were also reported by Cochrane and Peral \cite[footnote 1]{cochrane_02}).}
in 1994 provided the explicit expression for the constant term 
\be\notag
\sum_{l=1}^{n-1}  \csc\frac{\,\pi l\,}{n}\,< \,\frac{\,2n\,}{\pi}\!\left(\ln\frac{2n}{\pi}+\gamma \right)  +\!
\underbrace{\,2-\frac{1}{\,\pi\,}}_{1.681\ldots}\,,
\qquad n=2,3,4,\ldots\,,
\ee 
which was improved by Cochrane and Peral 
\be
\sum_{l=1}^{n-1}  \csc\frac{\,\pi l\,}{n}\,< \,\frac{\,2n\,}{\pi}\!\left(\ln\frac{2n}{\pi}+\gamma \right)  +\!
\underbrace{\,1-\frac{1}{\,\pi\,}}_{0.681\ldots}\,,
\qquad n=2,3,4,\ldots
\ee 
several years later \cite[Lemma 2]{cochrane_01}.
In 2003, Alzer and Koumandos \cite[Lemma 2]{alzer_01}
provided a more accurate estimate of $S_n$ 
\be
\alpha<  \sum_{l=1}^{n-1}\csc\frac{\,\pi l\,}{n} - \frac{\,2n\,}{\pi}\!\left(\ln\frac{2n}{\pi}+\gamma \right) \, <
\beta\,,\qquad n=2,3,4,\ldots
\ee
with the best possible constants $\,\alpha= 1 - 4(\gamma+2\ln2-\ln\pi)/\pi\,=-0.0425\ldots$ and $\,\beta=0\,.$
Another estimate for $S_n$ was proposed by Pomerance \cite[p.~536, Eq.~(5)]{pomerance_01} in 2011:
\be\label{08934cnu4}
\sum_{l=1}^{n-1}\csc\frac{\,\pi l\,}{n} \, < \frac{\,2n\,}{\pi}\left(\ln\frac{4n}{\pi}+\frac{\pi^2}{12n^2}\right)
\,,\qquad n=2,3,4,\ldots
\ee
An even more accurate estimate of $S_n$ was given by a team of Chinese researchers in 2023 \cite[Sec.~4, Lemma 3]{tong_01}:
\be\label{9872xychn2389}
\underbrace{-\frac{0.358}{\pi n}}_{-0.113\ldots n^{-1}}\!\!<
\sum_{l=1}^{n-1}  \csc\frac{\,\pi l\,}{n} - \frac{\,2n\,}{\pi}\!\left(\ln\frac{2n}{\pi}+\gamma \right) \, <\!\!\!
\underbrace{-\frac{0.186}{\pi n}}_{-0.059\ldots n^{-1}}\,,\qquad n=2,3,4,\ldots
\ee
Comparing this very recent result to our Theorem \ref{oiehrhg4w5h}, formula \eqref{8438394yrxn}, 
we see that the lower bound of the latter estimate $-0.113\ldots n^{-1}$ can actually be increased 
to $-0.0872\ldots n^{-1}$. Fig.~\ref{89723yxb29} compares estimates \eqref{08934cnu4} and 
\eqref{9872xychn2389} with our bounds provided by Theorem \ref{oiehrhg4w5h}.

\begin{figure}[!t]   
\centering
\includegraphics[width=0.8\textwidth]{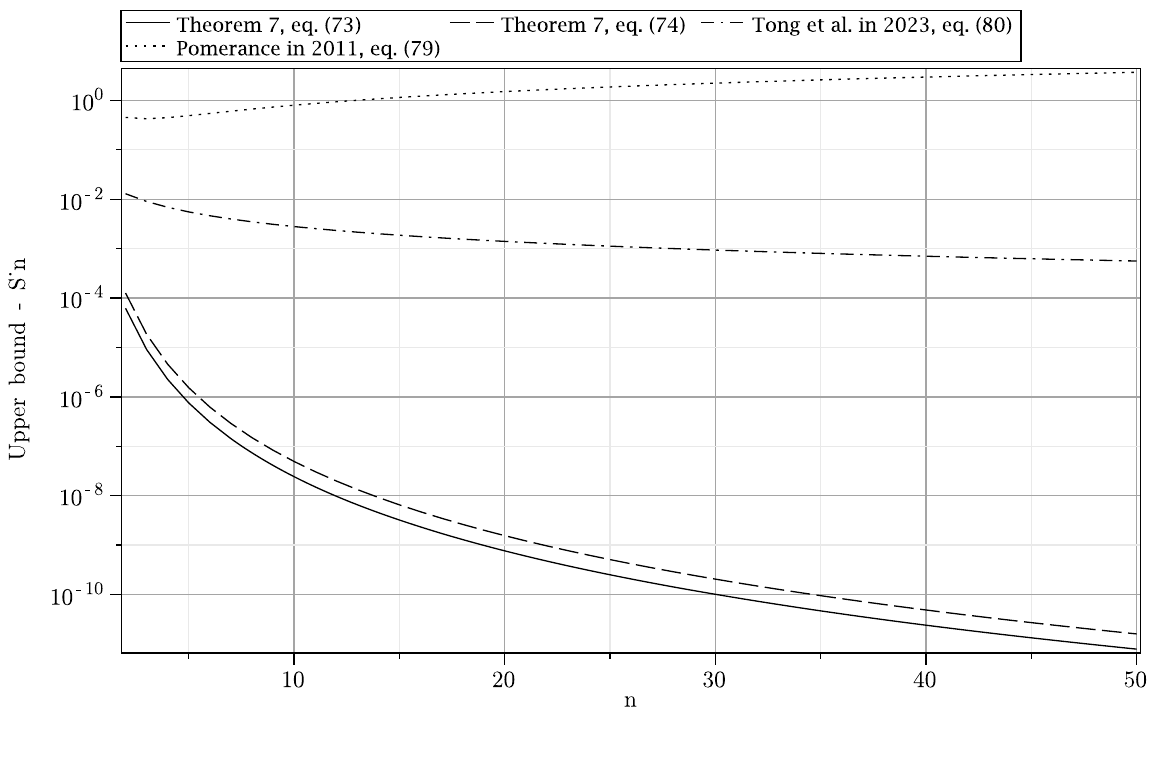}
\vspace{-3em}
\caption{The differences between four different upper bounds for $S_n$ and the value of $S_n$.
One may readily note that the bounds given by Theorem \ref{oiehrhg4w5h} (solid and dashed lines) are much more
accurate than \eqref{08934cnu4} obtained by Pomerance in 2011 (dotted line) and than \eqref{9872xychn2389} obtained
by Tong et al.~in 2023 (dash-dotted line). Moreover, in Theorem \ref{oiehrhg4w5h}, the bounds provided by inequality \eqref{o4uincf87nb4}, 
are more accurate than those provided by \eqref{8438394yrxn}. In particular, the difference between the upper bound given in \eqref{o4uincf87nb4}
and $S_n$ is lesser than $3\times10^{-8}$ for $n\geqslant10$ and lesser than $8\times10^{-12}$ for $n\geqslant50$ (the relative differences
are lesser than $2\times10^{-9}$ and $7\times10^{-14}$ respectively).}
\label{89723yxb29}
\end{figure}

\begin{theorem}[Bounds for $\bm{S_n(\varphi,a)}$]\label{090j19423unc}
For the sum $S_n(\varphi,a)$, $\,0<\varphi<\pi\,$, $\,1-\frac{\varphi}{\pi}<a<2-\frac{\varphi}{\pi}\,$,
$n=2,3,4,\ldots\,$, the following upper and lower bounds hold, depending on the choice of $\varphi$ and $a$: 
if $\,2\varphi +a\pi > 2\pi \,$, then
\be\notag
\begin{array}{ll}
\displaystyle 
\frac{A}{\! n\!}  -\,\frac{B}{\,n^3\,} < 
\sum_{l=1}^{n-1} \csc\!\left(\varphi+\frac{\,a\pi l\,}{n}\!\right) 
 -\:& \displaystyle \frac{n}{a}\ctg\frac{n(\pi-\varphi)}{a}
- \frac{n}{a\pi}\ln\frac{\,\tg\left(\pi - \frac12\pi a- \frac12 \varphi\right)\,}{\tg\frac12\varphi } +\\[7mm]
& \displaystyle\quad
+\frac{\,\csc\varphi+\csc(\varphi+a\pi)\,}{2}
<   \frac{A}{\! n\!}
\end{array}
\ee
where 
\be\notag
A\equiv \frac{a\pi}{12} \Big\{\!\csc\varphi\cdot\ctg\varphi - \csc(\varphi+a\pi)\ctg(\varphi+a\pi)\Big\}
\ee
and 
\be\notag
B\equiv \frac{a^3\pi^3}{720} \Big\{
\csc\varphi\cdot\ctg^3\varphi+5\csc^3\varphi\cdot\ctg\varphi 
- \csc(\varphi+a\pi) \ctg^3(\varphi+a\pi) - 5\csc^3(\varphi+a\pi)\ctg(\varphi+a\pi) 
\Big\} 
\ee
On the contrary, if $\,2\varphi +a\pi < 2\pi \,$, then
\be\notag
\begin{array}{ll}
\displaystyle 
\frac{A}{\! n\!} < 
\sum_{l=1}^{n-1} \csc\!\left(\varphi+\frac{\,a\pi l\,}{n}\!\right) 
 -\:& \displaystyle \frac{n}{a}\ctg\frac{n(\pi-\varphi)}{a}
- \frac{n}{a\pi}\ln\frac{\,\tg\left(\pi - \frac12\pi a- \frac12 \varphi\right)\,}{\tg\frac12\varphi }+ \\[7mm]
& \displaystyle\quad
+\frac{\,\csc\varphi+\csc(\varphi+a\pi)\,}{2}
<   \frac{A}{\! n\!} -\,\frac{B}{\,n^3\,} 
\end{array}
\ee
with the same $A$ and $B$ as above.
\end{theorem}

\begin{corollary}[Bounds for $\bm{S_n(\varphi,1)}$]
For the sum $S_n(\varphi,1)$, $\,0<\varphi<\pi\,$ and $\varphi\neq\pi k/n$, $k\in\mathbbm{Z}$,
$n=2,3,4,\ldots\,$, the following upper and lower bounds hold:
\be\notag
\begin{cases}
\displaystyle
\frac{A}{\! n\!}  -\,\frac{B}{\,n^3\,} < 
\sum_{l=1}^{n-1} \csc\!\left(\varphi+\frac{\,\pi l\,}{n}\!\right) 
+ n\ctg\varphi n
+ \frac{\,2n\,}{\pi}\ln\tg\frac\varphi2   
<   \frac{A}{\! n\!}	\,, \qquad& 0<\varphi<\frac12\pi   \\[8mm]
\displaystyle
\frac{A}{\! n\!} < 
\sum_{l=1}^{n-1} \csc\!\left(\varphi+\frac{\,\pi l\,}{n}\!\right) 
+ n\ctg\varphi n
+ \frac{\,2n\,}{\pi}\ln\tg\frac\varphi2   
<   \frac{A}{\! n\!}  -\,\frac{B}{\,n^3\,}\,, \qquad	& \frac12\pi<\varphi<\pi
\end{cases}
\ee
where 
\be\notag
A\equiv \frac{\, \pi \csc\varphi\cdot\ctg\varphi \,}{6}
\ee
and 
\be\notag
B\equiv \frac{\pi^3}{360} \Big\{ \!\csc\varphi\cdot\ctg^3\varphi+5\csc^3\varphi\cdot\ctg\varphi \Big\} 
\ee
\end{corollary}

\begin{proof}
By a line of reasoning similar to that used in the proof of the previous Theorem, we first study the sign of the general term 
of the expansion in Theorem \ref{iu2389ghi}:
\be\notag
\sgn\left[\frac{\,(a\pi)^{2r-1}B_{2r}\,}{\,n^{2r-1} (2r)!\,} 
\Big\{\mathscr{F}_{2r-1}(\varphi) - \mathscr{F}_{2r-1}(\varphi + a\pi)  \Big\} \right] =
(-1)^{r+1} \sgn\Big\{\mathscr{F}_{2r-1}(\varphi) - \mathscr{F}_{2r-1}(\varphi + a\pi)  \Big\}
\ee
since $\sgn B_{2r}=(-1)^{r+1}.$ Now, considering Euler's formula \eqref{hd2893dh2} for different intervals and evaluating
its $(2r-1)$th derivative
\be
\frac{d^{2r-1}\csc \alpha}{d\alpha^{2r-1}} \,=
\,\frac{(-1)^{n+1}}{\pi^{2r}}\int\limits_0^\infty \!\frac{\,x^{\frac\alpha\pi} \ln^{2r-1}x}{\,x^n(1+x)\,} \, dx\,,
\qquad (n-1)\pi<\alpha<n\pi\,, \quad n\in\mathbbm{N}\,,
\ee
we see at once that for $0<\varphi<\pi$ and $\pi<\varphi+\pi a<2\pi$
\be\notag
\begin{array}{ll}
\displaystyle & \displaystyle
\left.\frac{d^{2r-1}\csc \alpha}{d\alpha^{2r-1}}\right|_{\alpha=\varphi} -  
\left.\frac{d^{2r-1}\csc \alpha}{d\alpha^{2r-1}}\right|_{\alpha=\varphi+\pi a} 
\,=\,\frac{1}{\pi^{2r}}\int\limits_{-\infty}^{+\infty} \!\frac{\,e^{t\frac\varphi\pi} \, t^{2r-1}}{\,1+e^t\,} \, dt
+ \frac{1}{\pi^{2r}}\!\int\limits_{-\infty}^{+\infty} \!\frac{\,e^{t\frac\varphi\pi+at-t} \, t^{2r-1}}{\,1+e^t\,} \, dt =
\end{array}
\ee

\be\label{e3h498hc}
\begin{array}{ll}
\displaystyle & \displaystyle
\,=\,\frac{1}{\pi^{2r}}\!\int\limits_{0}^{\infty} \!
\left\{\sh\left[t\left(\frac\varphi\pi+a-1\right)\right] + \sh\left[t\left(\frac\varphi\pi+a-2\right)\right]
+\sh\left(\frac{\varphi t}{\pi}\right) + \sh\left[t\left(\frac\varphi\pi-1\right)\right]\right\}\frac{\, t^{2r-1} dt}{\,1+\ch t\,}\\[8mm]
\displaystyle & \displaystyle
\,=\,\frac{1}{\pi^{2r}}\!\int\limits_{0}^{\infty} 
\left\{\sh\left[t\left(\frac\varphi\pi - \frac12\right)\right] + \sh\left[t\left(\frac\varphi\pi+a-\frac32\right)\right] \right\}
\!\frac{\, t^{2r-1} dt\,}{\,\ch\frac12 t\,} \\[8mm]
\displaystyle & \displaystyle
\,=\,\frac{2}{\pi^{2r}}\!\int\limits_{0}^{\infty} 
\sh\left[t\left(\frac\varphi\pi + \frac{a}{2}-1\right)\right] \ch\left[t\left(\frac{a}{2}-\frac12\right)\right] 
\!\frac{\, t^{2r-1} dt\,}{\,\ch\frac12 t\,}.
\end{array}
\ee
Clearly, the sign of the latter integral depends only on the sign of the hyperbolic sine; therefore
\be\label{0948dnh3u4}
\sgn\left[ \frac{\,(a\pi)^{2r-1}B_{2r}\,}{\,n^{2r-1} (2r)!\,} 
\Big\{\mathscr{F}_{2r-1}(\varphi) - \mathscr{F}_{2r-1}(\varphi + a\pi)  \Big\} \right] =
\begin{cases}
(-1)^{r+1} \,,\quad& \text{if } \; 2\varphi +a\pi > 2\pi \\[1mm]
0 \,,\quad& \text{if } \; 2\varphi +a\pi = 2\pi \\[1mm]
(-1)^{r} \,,\quad &\text{if }  \; 2\varphi +a\pi < 2\pi
\end{cases}
\ee
Thus, the series from Theorem \ref{iu2389ghi} is enveloping. Setting $N=2$ and $N=3$ into this series, remarking that 
$\, \mathscr{F}_{1}(\varphi)=  -\csc\varphi\cdot\ctg\varphi\,$ and that
$\, \mathscr{F}_{3}(\varphi)=  -\csc\varphi\cdot\ctg^3\varphi - 5\csc^3\varphi\cdot\ctg\varphi \,$,
and accounting for the sign in accordance with \eqref{0948dnh3u4},
we obtain the required result. 
\end{proof}

\section{On a finite sums of secants}
\subsection{Preliminary remarks}
The sums of secants \eqref{984ycbn492v3} is related to that of cosecants \eqref{984ycbn492v2}
by a simple relationship: 
$$\,C_n(\varphi,a)\,=\,S_n\Big(\varphi+\tfrac12\pi,a\Big)\,.$$
Its properties, therefore, may be established without any difficulty from those of $S_n(\varphi,a)$.
Besised, we immediately remark that the case $C_n(0,1)$ is uninteresting since 
\be\notag
\pv\sum_{l=1}^{n-1}  \sec\frac{\,\pi l\,}{n} \,= 0
\ee

Below, we give some Theorems for $C_n(\varphi,a)$, which are analogous to those we previously proved for $S_n(\varphi,a)$.
Since they can be readily obtained by a simple shift of argument $\varphi$ by $\pi/2$, we supress the proofs.

\subsection{Integral representations}
\begin{lemma}[Improper integral representation]\label{mlemmav2}
The sum of secants \eqref{984ycbn492v3} may be represented via the following integral:
\be\notag
\sum_{l=1}^{n-1}  \sec\!\left(\varphi+\frac{\,a\pi l\,}{n}\!\right)= \,\frac{\,2n\,}{\pi}\!
\int\limits_0^\infty \!\frac{\,\sh\big[ax(n-1) \big] \,}{\sh ax \cdot\ch nx} \ch\!\left[nx\!\left(\frac{2\varphi}{\pi}+a\right)\right] dx\,,\qquad 
\ee
where $\,\displaystyle-\frac{a\pi}{n}-\frac{\pi}{2}<\Re\varphi <+\frac{\pi}{2}+\frac{a\pi}{n}-\pi a$\,.
\end{lemma}

\subsection{Series representations}
\begin{theorem}[Digamma infinite series representation]\label{coiue2ynxv2}
The function $C_n(\varphi,a)$ may be expanded into the infinite series involving the digamma functions only:
\be\label{c904cj9384nvc431u}
\begin{array}{ll}
\displaystyle 
\sum_{l=1}^{n-1}  \sec\!\left(\varphi+\frac{\,a\pi l\,}{n}\!\right) 
=\,\frac{\,n\,}{a\pi}\sum_{k=0}^{\infty} (-1)^k &\displaystyle \left\{  
\Psi\!\left(\frac{nk}{a}+n+\frac{n\varphi}{a\pi}+\frac{n}{2a}\right) - \Psi\!\left(1+\frac{nk}{a}-\frac{n\varphi}{a\pi}-n+\frac{n}{2a}\right) -  \right.\\[8mm]
&\displaystyle \quad\left.
- \Psi\!\left(1+\frac{nk}{a}+\frac{n\varphi}{a\pi}+\frac{n}{2a}\right) + \Psi\!\left(\frac{nk}{a} +\frac{n}{2a}-\frac{n\varphi}{a\pi}\right) \! \right\}\,,
\end{array}
\ee
where the parameters $\varphi$ and $a$ are chosen as stated in \eqref{984ycbn492v3}.
\end{theorem}

\begin{theorem}[Digamma finite series representations]\label{coiue2ynxv2}
The function $C_n(\varphi,a)$ may be represented by a finite series involving the digamma functions only:
\be\notag
\begin{array}{ll}
\displaystyle 
\sum_{l=1}^{n-1}  \sec\!\left(\varphi+\frac{\,a\pi l\,}{n}\!\right) 
=\,\frac{\,1\,}{2\pi}\sum_{l=1}^{n-1}  &\displaystyle \left\{  
\Psi\!\left(\frac34+\frac{al}{2n}+\frac{\varphi}{2\pi}\right) - \Psi\!\left(\frac14+\frac{al}{2n}+\frac{\varphi}{2\pi}\right) +  \right.\\[8mm]
&\displaystyle \quad\left.
+ \Psi\!\left(\frac34+\frac{al}{2n}-\frac{\varphi}{2\pi}-\frac{a}{2}\right) - \Psi\!\left(\frac14+\frac{al}{2n}-\frac{\varphi}{2\pi}-\frac{a}{2}\right) 
\! \right\}\,.
\end{array}
\ee
where $\varphi$ and $a$ should satisfy conditions stated earlier in \eqref{984ycbn492v3}.
\end{theorem}

\subsection{Asymptotic studies for large $n$}

\begin{theorem}[Asymptotics of $\bm{C_n(\varphi,a)}$]\label{iu2389ghiv2}
If, in addition to what was stated in \eqref{984ycbn492v3}, arguments $\varphi$ and $a$ are chosen so that
\be\label{024309if}
-\frac{\,\pi\,}{2}<\varphi<+\frac{\,\pi\,}{2}\qquad \text{and}\qquad \frac12-\frac{\varphi}{\pi}<a<\frac32-\frac{\varphi}{\pi}\,,
\ee
then there exists a sufficiently large $n_0$, such that for any $n>n_0$ the following equality holds:
\be\notag
\begin{array}{ll}
\displaystyle 
\sum_{l=1}^{n-1} \sec\!\left(\varphi+\frac{\,a\pi l\,}{n}\!\right) 
 =&\displaystyle \: \frac{n}{a}\ctg\frac{n\Big(\frac12\pi-\varphi\Big)}{a}
+ \frac{n}{a\pi}\ln\frac{\,\tg\left(\frac34\pi - \frac12\pi a- \frac12 \varphi\right)\,}{\tg\Big(\frac12\varphi+\frac14\pi\Big) } 
-\frac{\,\sec\varphi+\sec(\varphi+a\pi)\,}{2} -\\[8mm]
&\displaystyle 
- \sum_{r=1}^{N-1} \frac{\,(a\pi)^{2r-1}B_{2r}\,}{\,n^{2r-1} (2r)!\,} 
\Big\{\mathscr{G}_{2r-1}(\varphi) - \mathscr{G}_{2r-1}(\varphi + a\pi)  \Big\} 
+ O\big(n^{-2N+1}\big)\,,
\end{array}
\ee
where $N=2,3,4,\ldots\,,$ and
\be\notag
\mathscr{G}_{2r-1}(\alpha)\,\equiv\frac{d^{2r-1}\sec \alpha}{d\alpha^{2r-1}} \,.
\ee
At $n\to\infty$, the leading terms are placed from left to right in the first line;
the finite sum with the Bernoulli numbers tends to zero as $n\to\infty$.
\end{theorem}

\noindent\textbf{Remark.}
Similarly to what was noted for $\mathscr{F}_{2r-1}(\alpha)$, one may also show that 
\be\notag
\mathscr{G}_{2r-1}(\alpha) \,=\,
\frac{1}{\,(2\pi)^{2r}\,}\left\{ \Psi_{2r-1}\left(\frac34+\frac{\alpha}{2\pi}\right) 
+ \Psi_{2r-1}\left(\frac14-\frac{\alpha}{2\pi}\right)
-\Psi_{2r-1}\left(\frac14+\frac{\alpha}{2\pi}\right) -  \Psi_{2r-1}\left(\frac34-\frac{\alpha}{2\pi}\right)\!\right\}.
\ee
We also have for the difference
\be\notag
\mathscr{G}_{2r-1}(\varphi) - \mathscr{G}_{2r-1}(\varphi + a\pi)
\,=\,\frac{2^{2r+1}}{\pi^{2r}}\!\int\limits_{0}^{\infty} 
\!\frac{\, \ch\big[t(a-1)\big] \,}{\,\ch t\,} 
\sh\!\left[t\!\left(\frac{2\varphi}{\pi} + a-1\right)\right] t^{2r-1}\,  dt\,,
\ee
provided that conditions \eqref{024309if} are fulfilled.

\begin{corollary}[Asymptotics of $\bm{C_n(\varphi,1)}$]\label{9uricxn34089v2}
If $\,-\frac12\pi<\varphi<+\frac12\pi\,$ and $\varphi\neq\pi k/n$, $k\in\mathbbm{Z}$, then, 
for $\,n>n_0\,$ the sum $C_n(\varphi,1)$ admits the following (asymptotic) expansion:
\be\notag
\sum_{l=1}^{n-1} \sec
\!\left(\varphi+\frac{\,\pi l\,}{n}\!\right) 
 = \,-n\ctg\left(\varphi n+\frac{\pi n}{2}\right)
- \frac{\,2n\,}{\pi}\ln\tg\left(\frac\varphi2+\frac\pi4\right)  
- \,2\!\sum_{r=1}^{N-1} \frac{\,\pi^{2r-1}B_{2r}\,}{\,n^{2r-1} (2r)!\,} 
\mathscr{G}_{2r-1}(\varphi) 
+ O\big(n^{-2N+1}\big)
\ee
with the same $\mathscr{G}_{2r-1}(\alpha)$, $N$ and $n_0$ as in Theorem \ref{iu2389ghiv2}.
At $n\to\infty$, first two terms in this expansion are leading, while the last sum with the Bernoulli numbers 
is $o(1)$.
\end{corollary}

\begin{corollary}[Asymptotics of $\bm{C_n(0,a)}$]\label{9uricxn34089v2}
If $\,-\nicefrac12<a<+\nicefrac32\,$, $a\neq n/(2k)$, $k\in\mathbbm{Z}$, then
for $\,n>n_0\,$ the sum $C_n(0,a)$ admits the following (asymptotic) expansion:
\be\notag
\begin{array}{ll}
\displaystyle 
\sum_{l=1}^{n-1} \sec\frac{\,a\pi l\,}{n}
 =&\displaystyle \:  \lfloor a +\tfrac12\rfloor\frac{n}{a}\ctg\frac{\pi n}{2a}
+ \frac{n}{a\pi}\ln\left|\tg\!\left(\tfrac{3}{4}\pi - \tfrac{1}{2}a\pi\right)\right|
-\,\frac{\,1+\sec a\pi\,}{2} +\\[8mm]
&\displaystyle 
+ \sum_{r=1}^{N-1} \frac{\,(a\pi)^{2r-1}B_{2r}\,}{\,n^{2r-1} (2r)!\,} 
 \,\mathscr{G}_{2r-1}( a\pi) 
+ O\big(n^{-2N+1}\big)\,,
\end{array}
\ee
with the same $\mathscr{G}_{2r-1}(\alpha)$, $N$ and $n_0$ as in Theorem \ref{iu2389ghiv2}.
At $n\to\infty$, first three terms in this expansion are leading, while the last sum with the Bernoulli numbers 
is $o(1)$.
\end{corollary}

\subsection{Bounds for $C_n(\varphi,a)$}
Bounds for $C_n(\varphi,a)$ immediately follow from those for $S_n(\varphi,a)$
by shifting the angle $\varphi$ by $\tfrac12\pi$. 

\begin{theorem}[Bounds for $\bm{C_n(\varphi,a)}$]
For the sum $S_n(\varphi,a)$, $\,-\frac12\pi<\varphi<+\frac12\pi \,$, $\,\frac12-\frac{\varphi}{\pi}<a<\frac32-\frac{\varphi}{\pi}\,$,
$n=2,3,4,\ldots\,$, the following upper and lower bounds hold, depending on the choice of $\varphi$ and $a$: 
if $\,2\varphi +a\pi > 2\pi \,$, then
\be\notag
\begin{array}{ll}
\displaystyle 
\frac{A}{\! n\!}  -\,\frac{B}{\,n^3\,} < 
\sum_{l=1}^{n-1} \sec\!\left(\varphi+\frac{\,a\pi l\,}{n}\!\right) 
 -\:& \displaystyle \frac{n}{a}\ctg\frac{n\Big(\frac12\pi-\varphi\Big)}{a}
- \frac{n}{a\pi}\ln\frac{\,\tg\left(\frac34\pi - \frac12\pi a- \frac12 \varphi\right)\,}{\tg\Big(\frac12\varphi+\frac14\pi\Big) } + \\[7mm]
& \displaystyle\quad
+\frac{\,\sec\varphi+\sec(\varphi+a\pi)\,}{2}
<   \frac{A}{\! n\!}
\end{array}
\ee
where 
\be\notag
A\equiv -\frac{a\pi}{12} \Big\{\!\sec\varphi\cdot\tg\varphi - \sec(\varphi+a\pi)\tg(\varphi+a\pi)\Big\}
\ee
and 
\be\notag
B\equiv -\frac{a^3\pi^3}{720} \Big\{
\sec\varphi\cdot\tg^3\varphi+5\sec^3\varphi\cdot\tg\varphi 
- \sec(\varphi+a\pi) \tg^3(\varphi+a\pi) - 5\sec^3(\varphi+a\pi)\tg(\varphi+a\pi) 
\Big\} 
\ee
On the contrary, if $\,2\varphi +a\pi < 2\pi \,$, then
\be\notag
\begin{array}{ll}
\displaystyle 
\frac{A}{\! n\!} < 
\sum_{l=1}^{n-1} \sec\!\left(\varphi+\frac{\,a\pi l\,}{n}\!\right) 
 -\:& \displaystyle \frac{n}{a}\ctg\frac{n\Big(\frac12\pi-\varphi\Big)}{a}
- \frac{n}{a\pi}\ln\frac{\,\tg\left(\frac34\pi - \frac12\pi a- \frac12 \varphi\right)\,}{\tg\Big(\frac12\varphi+\frac14\pi\Big) }  + \\[7mm]
& \displaystyle\quad
+\frac{\,\sec\varphi+\sec(\varphi+a\pi)\,}{2}
<   \frac{A}{\! n\!} -\,\frac{B}{\,n^3\,} 
\end{array}
\ee
with the same $A$ and $B$ as above.
\end{theorem}

\begin{corollary}[Bounds for $\bm{C_n(\varphi,1)}$]
For the sum $S_n(\varphi,1)$, $\,-\frac12\pi<\varphi<+\frac12\pi \,$ and $\varphi\neq\pi \big(\nicefrac{k}{n}-\nicefrac12\big)$, $k\in\mathbbm{Z}$,
$n=2,3,4,\ldots\,$, the following upper and lower bounds hold:
\be\notag
\begin{cases}
\displaystyle
\frac{A}{\! n\!}  -\,\frac{B}{\,n^3\,} < 
\sum_{l=1}^{n-1} \sec\!\left(\varphi+\frac{\,\pi l\,}{n}\!\right) 
+n\ctg\left(\varphi n+\frac{\pi n}{2}\right)
+ \frac{\,2n\,}{\pi}\ln\tg\left(\frac\varphi2+\frac\pi4\right)   
<   \frac{A}{\! n\!}	\,, \qquad& -\frac12\pi<\varphi<0\\[8mm]
\displaystyle
\frac{A}{\! n\!} < 
\sum_{l=1}^{n-1} \sec\!\left(\varphi+\frac{\,\pi l\,}{n}\!\right) 
+n\ctg\left(\varphi n+\frac{\pi n}{2}\right)
+ \frac{\,2n\,}{\pi}\ln\tg\left(\frac\varphi2+\frac\pi4\right) 
<   \frac{A}{\! n\!}  -\,\frac{B}{\,n^3\,}\,, \qquad	& 0<\varphi<\frac12\pi  
\end{cases}
\ee
where 
\be\notag
A\equiv -\frac{\, \pi \sec\varphi\cdot\tg\varphi \,}{6}
\ee
and 
\be\notag
B\equiv -\frac{\pi^3}{360} \Big\{ \!\sec\varphi\cdot\tg^3\varphi+5\sec^3\varphi\cdot\tg\varphi \Big\} .
\ee
\end{corollary}

\section{On some other sums and series related to $S_n(\varphi,a)$ and $C_n(\varphi,a)$}
All the formul\ae~that we obtained for $S_n(\varphi,a)$ and $C_n(\varphi,a)$ may legitimately be differentiated 
or integrated with respect to $\varphi$ and $a$.
Furthermore, they can also be summed with respect to $a$, $\varphi$ and $n$, and also be combined with other 
summation formul\ae. These procedures permit to obtain
many useful relationships, as well as help to derive some important properties. For instance, the expansion obtained
in Corollary \ref{9uricxn34089} may be combined with the cotangent summation Theorem \ref{kjcwibaz2}, which 
immediately yields:
\be\label{42390cj34urf}
\sum_{l=1}^{n-1} \ctg\!\left(\!\varphi+\frac{\,\pi l\,}{2n}\!\right) 
 = \,-\frac{\,2n\,}{\pi}\ln\tg\varphi 
-  \ctg2\varphi
- \,2\!\sum_{r=1}^{N-1} \frac{\,\pi^{2r-1}B_{2r}\,}{\,n^{2r-1} (2r)!\,} 
\mathscr{F}_{2r-1}(2\varphi) 
+ O\big(n^{-2N+1}\big)
\ee
with the same $\mathscr{F}_{2r-1}(\alpha)$, $N$ and $n_0$ as in Theorem \ref{iu2389ghi},
and $\,0<\varphi<\frac12\pi\,$, $\varphi\neq\pi k/n$, $k\in\mathbbm{Z}$.
Shifting $\varphi$ by $\frac12\pi$, we also obtain
\be\notag
\sum_{l=1}^{n-1} \tg\!\left(\!\varphi+\frac{\,\pi l\,}{2n}\!\right) 
 = \,-\frac{\,2n\,}{\pi}\ln\big|\tg\varphi\big| 
+  \ctg2\varphi
- \,2\!\sum_{r=1}^{N-1} \frac{\,\pi^{2r-1}B_{2r}\,}{\,n^{2r-1} (2r)!\,} 
\mathscr{F}_{2r-1}(2\varphi) 
+ O\big(n^{-2N+1}\big)
\ee
under the same conditions as above and 
provided that $\,\frac12\pi<\varphi<\pi\,$, $\varphi\neq\pi k/n$, $k\in\mathbbm{Z}$.
Furthermore, evaluating the derivative of \eqref{42390cj34urf} with respect to $\varphi$, we immediately obtain
\be\label{233ficm4}
\sum_{l=1}^{n-1} \csc^2\!\left(\!\varphi+\frac{\,\pi l\,}{2n}\!\right) 
 = \,\frac{\,4n\,}{\pi}\csc2\varphi 
-  2\csc^22\varphi
+ \,4\!\sum_{r=1}^{N-1} \frac{\,\pi^{2r-1}B_{2r}\,}{\,n^{2r-1} (2r)!\,} 
\mathscr{F}_{2r}(2\varphi) 
+ O\big(n^{-2N+1}\big)
\ee
with the same $\mathscr{F}_{2r-1}(\alpha)$, $N$ and $n_0$ as in Theorem \ref{iu2389ghi},
and $\,0<\varphi<\frac12\pi\,$, $\varphi\neq\pi k/n$, $k\in\mathbbm{Z}$.
Similar result can be obtained for the sum of a series of squares of secants.
Note that the above sums behave much ``smoother'' than $S_n(\varphi,a)$ and exhibit only a linear growth with $n$,
because the ``disturbing'' term $\ctg n\varphi$ is absent. It is interesting that a similar formula obtained 
by Euler, see \eqref{89ybwetv3}, has a qualitatively different behaviour and again contains the ``disturbing'' term,
which is in this case $n^2\csc^2 n\varphi.$

\small
\bibliographystyle{crelle}

\end{document}